\documentclass{amsart}
\usepackage[a4paper,top=4cm,bottom=4cm,left=2.5cm,right=2.5cm,bindingoffset=5mm]{geometry}

\usepackage{amssymb,bm,amsmath,amsthm}
\usepackage{graphicx,mathtools,mathrsfs,subfig,hyperref}

\def\RR{\mathbb R}

\def\dt{\Delta t}
\def\nt{N_{t}}
\newtheorem{Theorem}{Theorem}
\newtheorem{Proposition}{Proposition}
\newtheorem{Lemma}{Lemma}

\makeatletter
\@namedef{subjclassname@2020}{\textup{2020} Mathematics Subject Classification}
\makeatother

\title{Effects of fractional derivatives in epidemic models}

\author[C. Balzotti, M. D'Ovidio, A.C. Lai, P. Loreti]{Caterina Balzotti$^1$\email{c.balzotti@iac.cnr.it}, Mirko D'Ovidio$^2$\email{mirko.dovidio@uniroma1.it}, Anna Chiara Lai$^2$\email{annachiara.lai@uniroma1.it}, Paola Loreti$^2$\email{paola.loreti@uniroma1.it}}
 

\begin{document}
\maketitle
{\centering{%
\small $^{1}$   Istituto per le Applicazioni del Calcolo,\\ Consiglio Nazionale delle Ricerche, Rome, Italy; }\\
{\small $^{2}$  \small 
Dipartimento 
 di Scienze di Base e Applicate per l'Ingegneria, \\Sapienza Università di {Roma}, Rome, Italy\\}}

\begin{abstract}{We study  epidemic Susceptible-Infected-Susceptible  models in the fractional setting. The novelty is to consider models in which the susceptible and infected populations evolve according to different fractional orders. We study a model based on Caputo derivative, for which we establish existence results of the solutions. Also, we investigate a model based on Caputo-Fabrizio operator, for which we provide existence of solutions and a study of the equilibria. Numerical simulations for both models and a direct numerical comparison are also provided. 
}\end{abstract}

\section{Introduction}
The interest of the scientific community in mathematical modeling for epidemiology has grown considerably in recent years.
 The study of epidemic models began in the early 1900s with the pioneering work of Kermack and McKendrick \cite{KerMcK}. Their idea was to divide the population into groups which distinguish the individuals based on their status with respect to the infection, giving rise to the compartmental modeling for epidemics. The evolution in time of the disease is then described by a system of ordinary differential equations for each considered class.

In this work we focus on the SIS (Susceptible-Infected-Susceptible) model \cite{Hethcote}, describing infections which do not confer immunity to recovery from illness, such as influenza and common cold. 
Such theory describes compartmental models, where the population is divided into groups depending on the state of individuals, that is with respect to disease, distinguishing infected, susceptible. The use of mathematical models for epidemiology is useful to predict the behavior of an infection and take strategic decisions in emergency situations to limit the spread of the disease which is microscopically modeled by the fractional order of the derivative.

In recent years, the use of fractional derivatives for epidemic models has grown widely. The main advantage of fractional calculus is that it can incorporate memory effects into the model. Moreover, fractional models have an extra degree of freedom compared to classical models, which is particularly useful for fitting real data when available. We refer to \cite{chen2021AMM} for a recent review of fractional epidemic models.

In this paper we consider two fractional SIS models. One model is based on Caputo derivative, for which we establish existence results of the solutions and provide numerical simulations. The novelty is to allow the susceptible and infected population evolve according to different fractional orders. The other model is based on Caputo-Fabrizio operator. Here we let the susceptible population evolve according to the Caputo-Fabrizio fractional operator, whereas the infected population dynamics is based on ordinary differential equation. In this case, we rewrite the system as a system of ordinary differential equations, we study the equilibria and present some simulations. 

More precisely, let $\alpha, \alpha_1, \alpha_2 \in (0,1)$. We consider the initial-value problem for Caputo derivative with different orders
\begin{equation}\label{modello3I}
	\begin{dcases}
		{^C D}^{\alpha_{1}}_{t} S(t) = -\beta\frac{S(t)}{S(t)+I(t)}I(t)+\gamma I(t)\\
		{^C D}^{\alpha_{2}}_{t} I(t) = \beta\frac{S(t)}{S(t)+I(t)}I(t)-\gamma I(t)\\
		S(0)=S_0\\
I(0)=I_0.
	\end{dcases}
\end{equation}
and the initial-value problem for Caputo-Fabrizio operator
\begin{equation}\label{SIS-CF}
\begin{cases}
{^{CF} D}^\alpha_t S= -\left(\frac{\beta}{S+I} S-\gamma\right)I\\
I'= \left(\frac{\beta}{S+I} S-\gamma\right)I\\
S(0)=S_0\\
I(0)=I_0.
\end{cases}
\end{equation}
where $\beta,\gamma >0$, $S_0,I_0\geq 0$.

Here, in the formula \eqref{modello3I} above we denote the Caputo derivative \cite{Caputo} by 
\begin{align*}
{^C D}^{\alpha_i}_t u(t) = \frac{1}{\Gamma(1-\alpha_i)} \int_0^t \frac{d u}{d s}(s)\, (t-s)^{-\alpha_i}\, ds, \quad i=1,2
\end{align*}
where $\Gamma(\cdot)$ is the gamma function defined as
\begin{align*}
\Gamma(z) = \int_0^\infty e^{-s} s^{z-1}\, ds, \quad z>0.
\end{align*}
The Caputo derivative is well-defined for a function 
\begin{align*}
u \in C(0, T),\, u^\prime(s) (t-s)^{-\alpha_i} \in L^1((0, t)), \quad \forall t \in (0, T), \quad i=1,2.
\end{align*}

The Caputo-Fabrizio operator in \eqref{SIS-CF} is defined by
\begin{align*}
{^{CF}D}^\alpha_t u(t) = \frac{M(\alpha)}{1-\alpha}\int_a^t u'(\tau) e^{-\frac{\alpha}{1-\alpha}(t-\tau)}d\tau
\end{align*}
where $M(\alpha)$ is a non-negative scaling factor satisfying $M(0)=M(1)=1$

The case $\alpha_1=\alpha_2=1$ in \eqref{modello3I} has been extensively studied in the literature. Beginning with the models introduced in \cite{v1,v2,v3} in which the logistic equation is used to model population dynamics, a numerous researchers work in the field.   

The problem to find a solution to the fractional logistic
equation attracted many authors and many works on this topic have been written only providing approximations for that solution. A contribution to this discussion is given in \cite{DovLor18} based on a series representation of the solution which involves Euler’s
numbers. However, this approach is not applicable in this new context since it is based on the following property of the Caputo derivative 
\begin{align*}
\label{propertyCD}
{^C D}^{\alpha}_t u(t) = 0 \quad \textrm{if and only if} \quad \ u(t)=const \quad \forall t \geq 0, \quad \alpha \in (0,1).
\end{align*}
Indeed, our effort is to study the fractional indices of derivation in which $\alpha_1$ may differ from $\alpha_2$. We mainly study existence of solution and we propose numerical test based on the method presented in \cite{giga2019AA} that we illustrate by several pictures. The numerical tests are in agreement with the theory developed in the present paper, and they offer new perspectives (e.g. symmetries emerging in the evolution of the total population $N(t)=I(t)+S(t)$) for future works.

This is not only a challenging problem from the mathematical point of view since it is of interest in many applications in which we may act in different ways in order to slowdown the process.

Recently the Caputo-Fabrizio operator \cite{CapFab} has attracted many researchers. The peculiarity of this fractional operator is the presence of non-singular kernel in contrast with the Caputo derivative in which singular kernel appears in the definition. However, the Caputo-Fabrizio operator can be used to model processes with memory.

Our work is to consider \eqref{SIS-CF} in which we are able to reduce the problem to an ordinary differential equation. Here we obtain the solution of the differential equation and their equilibria. The analysis is concluded by numerical simulations. 
{Finally, in order to point out the effects of the Caputo and the Caputo-Fabrizio differential operators,  we show a direct comparison between the evolutions of the numerical solutions to the system \eqref{modello3I} and the system \eqref{SIS-CF}. }

 We conclude with the plan of the paper. In Section \ref{s2} we introduce the SIS model with fractional Caputo derivatives with different orders. In Section \ref{s21}, we investigate the existence of solutions. To this end, we adopt a constructive approach whose ideas are borrowed from the Carathéodory existence theorem for ordinary differential equations, see \cite[Chapter 2]{ODE}. In Section \ref{sec:CCnum} we complete the study with some numerical simulation. In Section \ref{s3} we address a SIS model with fractional Caputo-Fabrizio operator. In Section \ref{s31} we establish existence of solution by rewriting the system as a system of ordinary differential equations, and we characterize the associated equilibria. Section \ref{s32} is devoted to numerical simulations, also including numerical tests directly comparing the proposed Caputo and Caputo-Fabrizio models. Finally in Section \ref{s4} we draw our conclusions.

\section{Caputo fractional epidemic models}\label{s2}
{
In this section we investigate the properties of the Caputo fractional SIS model with different fractional orders \eqref{modello3I}, which we rewrite here for reader's convenience. 
\begin{equation}\label{modello3}
	\begin{dcases}
		{^C D}^{\alpha_{1}}_{t} S(t) = -\beta\frac{S(t)}{S(t)+I(t)}I(t)+\gamma I(t)\\
		{^C D}^{\alpha_{2}}_{t} I(t) = \beta\frac{S(t)}{S(t)+I(t)}I(t)-\gamma I(t)\\
		S(0)=S_0\\
I(0)=I_0.
	\end{dcases}
\end{equation}
%
with $\alpha_1,\alpha_2\in[0,1]$. The particular case of fractional SIS models of Caputo type with $\alpha_{1}=\alpha_{2}$ has been studied in several works, see e.g.\  \cite{balzotti2020FF, elsaka2013MSL, hassouna2018CSF} and references therein. The novelty of our work is the use of different fractional indices. This approach has been already proposed in \cite{li2019NA}, where the authors study an inverse problem to calibrate the parameters for the dengue fever. The  results obtained fit well the real data, suggesting that mixed order fractional epidemic models are needed in applications.	
}
%
%

 Here we establish existence results for the solutions of \eqref{modello3} and we present some related numerical simulations.	
	
We introduce the notation 
$$f(x,y):=	\left(\gamma-\beta \frac{x}{x+y}\right)y$$
so that above system \eqref{modello3} rewrites
	$$\begin{dcases}
		{^C D}^{\alpha_{1}}_{t} S(t) = f(S(t), I(t))\\
		{^C D}^{\alpha_{2}}_{t} I(t) = -f(S(t),I(t))\\
		S(0)=S_0\\
I(0)=I_0.
	\end{dcases}$$
\subsection{Solutions}\label{s21}
The existence of a {solution to \eqref{modello3} in the time interval $[0,T)$ (with $T\leq +\infty$) with (positive) initial data $(S_0,I_0)$} is implied by the existence of a couple of absolutely continuous functions $(S(t),I(t))$ satisfying the Volterra-type fractional integral equation
	\begin{equation}\label{soldef}\begin{split}
	S(t)=S_0+\frac{1}{\Gamma(\alpha_1)}\int_0^{t} f(S(s),I(s))(t-s)^{\alpha_1-1}ds\quad\text{and}\quad \\
	I(t)=I_0-\frac{1}{\Gamma(\alpha_2)}\int_0^{t} f(S(s),I(s))(t-s)^{\alpha_2-1}ds.\end{split}\end{equation}
for all  $t\in(0,T)$. 
Fix $\alpha_1,\alpha_2\in[0,1]$. In the following we make use of the following function
$$G(t;\alpha_1,\alpha_2):=\max\left\{\frac{t^{\alpha_i}}{\Gamma(\alpha_i+1)}|i=1,2\right\}$$
which satisfies
$$G(t;\alpha_1,\alpha_2)\geq\frac{t^{\alpha_i}}{\Gamma(\alpha_i+1)}= \int_0^t \frac{(t-s)^{\alpha_i-1}}{\Gamma(\alpha_i)}ds,\qquad \forall t>0,\, i=1,2.$$

 In the next two results we establish some invariance properties for the field $f$ in the cases $\gamma\geq \beta$ and $\gamma<\beta$, respectively. 
 
 Note that below we often make use of the symbols $(x_0,y_0)$ to denote couples of positive real numbers and $x(t),y(t)$ to denote continuous real valued functions. This choice is meant to lighten the notation and we point out that,  in the search of solutions of \eqref{modello3}, $(x_0,y_0)$ plays the role of initial data $(S_0,I_0)$ whereas   $x(t)$ and $y(t)$ represent the (approximated) evolution of the susceptible population $S(t)$ and infected population $I(t)$, respectively. 

\begin{Lemma}\label{l1}
Let $\alpha_1,\alpha_2\in[0,1]$ and let $0<\beta\leq \gamma$ and fix $\varepsilon\in(0,1)$. Let $T=T_{\varepsilon,\gamma}$ be the positive real number such that 
$$G(T;\alpha_1,\alpha_2)(1+\varepsilon)\gamma\leq\varepsilon.$$
Fix $x_0,y_0>0$ and define 
$$Q_{x_0,y_0}:=[x_0,x_0+\varepsilon y_0]\times[(1-\varepsilon)y_0,(1+\varepsilon)y_0].$$
Then for any couple of absolutely continuous functions $x,y$ such that $(x(t),y(t))\in Q_{x_0,y_0}$ for all $t\in[0,T]$,  the associated functions
$$\tilde x(t):=x_0+\int_0^t f(x(s),y(s))\frac{(t-s)^{\alpha_1-1}}{\Gamma(\alpha_1)}ds$$
$$\tilde y(t):=y_0-\int_0^t f(x(s),y(s))\frac{(t-s)^{\alpha_2-1}}{\Gamma(\alpha_2)}ds.$$
also satisfy $(\tilde x(t),\tilde y(t))\in Q_{x_0,y_0}$ for all $t\in[0,T]$.
\end{Lemma}
\begin{proof}
Let $(x(t),y(t))\in Q_{x_0,y_0}$ for all $t\in[0,T]$, In particular one has $x(t )>0$ and $y(t)\in ((1-\varepsilon)y_0,(1+\varepsilon)y_0)$ for all $t\in[0,T]$. Then 
$$f(x(s),y(s))=\left(\gamma-\beta \frac{x(s)}{x(s)+y(s)}\right)y(s)<\gamma y(s)\qquad \forall s\in[0,T]$$ 
and this implies
\begin{align*}
\tilde x(t)&=x_0+\int_0^t f(x(s),y(s))\frac{(t-s)^{\alpha_1-1}}{\Gamma(\alpha_1)}ds\leq x_0+\int_0^t \gamma y(s)\frac{(t-s)^{\alpha_1-1}}{\Gamma(\alpha_1)}ds\\
&\leq x_0+ \gamma(1+\varepsilon)y_0\int_0^t \frac{(t-s)^{\alpha_1-1}}{\Gamma(\alpha_1)}ds\leq x_0+ G(t)(1+\varepsilon)\gamma y_0 \leq x_0 +\varepsilon y_0.
\end{align*}
Since $\gamma\geq \beta$ then $f(x,y)>0$ for all $x,y>0$, therefore $\tilde x(t)>x_0>0$ for all $t\in [0,T]$ and this proves that $\tilde x(t)\in(0,x_0+\varepsilon y_0)$ for all $t\in[0,T]$. 
On the other hand,
\begin{align*}
|\tilde y(t)-y_0|&\leq\int_0^t |f(x(s),y(s))|\frac{(t-s)^{\alpha_2-1}}{\Gamma(\alpha_2)}ds< \int_0^t \gamma y(s)\frac{(t-s)^{\alpha_2-1}}{\Gamma(\alpha_2)}ds\\
&\leq \gamma (1+\varepsilon)y_0\int_0^t \frac{(t-s)^{\alpha_2}}{\Gamma(\alpha_2)}ds\leq G(t) (1+\varepsilon)\gamma y_0 \leq \varepsilon y_0
\end{align*}
and this concludes the proof. 
\end{proof}
The next result deals with the case $\beta>\gamma$ and it posits some invariance properties of $f$, similar to those in Lemma \ref{l1}, in a time interval $[0,T]$. The main difference with Lemma \ref{l1} is that in this case the time $T$ depends not only on system parameters ($\beta\leq\gamma,\alpha_1$ and $\alpha_2$) but also on initial data.

\begin{Lemma}\label{l2}
Let $\alpha_1,\alpha_2\in[0,1]$ and $\beta> \gamma>0$. Fix $x_0,y_0>0$ and let  
$\varepsilon=\varepsilon(x_0,y_0):=\min\{\frac{1}{2},\frac{1}{2}\frac{x_0}{y_0}\}$ so that $\varepsilon \in(0,\frac{1}{2}]$ and  
$$ x_0>\frac{1}{2}x_0\geq \varepsilon y_0.$$
Let $T=T_{\beta,\gamma,x_0,y_0}$ be such that 
\begin{equation}\label{T0}G(T;\alpha_1,\alpha_2)(\beta+\gamma)(1+\varepsilon)\leq \varepsilon\end{equation}
and define 
{\small$$Q_{x_0,y_0}:=  [ x_0-\varepsilon y_0, x_0+\varepsilon]\times [(1-\varepsilon)y_0,(1+\varepsilon)y_0]. $$}
Then for any couple of  continuous functions $x(t),y(t)$ such that $(x(t),y(t))\in Q_{x_0,y_0}$ for all $t\in[0,T]$,  the associated functions
$$\tilde x(t):=x_0+\int_0^t f(x(s),y(s))\frac{(t-s)^{\alpha_1-1}}{\Gamma(\alpha_1)}ds$$
$$\tilde y(t):=y_0-\int_0^t f(x(s),y(s))\frac{(t-s)^{\alpha_2-1}}{\Gamma(\alpha_2)}ds.$$
also satisfy $(\tilde x(t),\tilde y(t))\in Q_{x_0,y_0}$ for all $t\in[0,T]$.
\end{Lemma}
\begin{proof}
Fix $x_0,y_0>0$. Let $x(t),y(t)$ be two continuous functions satisfying $(x(t),y(t))\in Q_{x_0,y_0}$ for all $t\in[0,T]$ with $T=T_{\beta,\gamma,x_0,y_0}$ satisfying \eqref{T0}.
Since 
$$|f(x(s),y(s)|\leq (\beta+\gamma)|y(s)|\qquad \forall s\in[0,T],$$
then one has
\begin{align*}
|\tilde x(t)-x_0|&=|\int_0^tf(x(s),y(s))\frac{(t-s)^{\alpha_1-1}}{\Gamma(\alpha_1)}ds|\leq \int_0^t(\gamma+\beta)|y(s)|\frac{(t-s)^{\alpha_1-1}}{\Gamma(\alpha_1)}ds|\\
&\leq G(T;\alpha_1,\alpha_2)(\beta+\gamma)(1+\varepsilon)y_0\leq \varepsilon y_0. \end{align*}
This implies $\tilde x(t)\in(x_0-\varepsilon y_0,x_0+\varepsilon y_0)$ for all $t\in[0,T]$. 
Arguing as above, one also deduces
\begin{align*}
|\tilde y(t)-y_0|&=|\int_0^tf(x(s),y(s))\frac{(t-s)^{\alpha_2-1}}{\Gamma(\alpha_2)}ds|\leq \varepsilon y_0. \end{align*}
and this implies $\tilde y(t)\in[(1-\varepsilon)y_0,(1+\varepsilon)y_0]$ for all $t\in[0,T]$. 
\end{proof}

We are finally in position to state the main result of the section, namely a global existence for the solutions of \eqref{modello3} in the case $\gamma\geq \beta$ and a local existence result in the case $\beta>\gamma$. 
\begin{Theorem}
Let $\alpha_1,\alpha_2\in[0,1]$ and let $0<\beta\leq \gamma$. Then for every initial data $S_0,I_0>0$ the system \eqref{modello3} admists a positive solution in $(0,+\infty)$.

If otherwise $0<\gamma<\beta$ then for every initial data  $S_0,I_0>0$ the system \eqref{modello3} admists a positive solution in $(0,T]$ where $T=T_{x_0,y_0,\beta,\gamma}>0$ satisfies
\begin{equation}\label{Test}\max\left\{ \frac{T^{\alpha_1}}{\Gamma(\alpha_1)}, \frac{T^{\alpha_2}}{\Gamma(\alpha_2)}\right\}=G(T;\alpha_1,\alpha_2)= \begin{cases}
\frac{1}{3(\beta+\gamma)}&\text{if } S_0\geq I_0;\\
\frac{S_0}{S_0+2I_0}\frac{1}{(\beta+\gamma)}&\text{if } S_0<I_0.
\end{cases}\end{equation}
\end{Theorem}
\begin{proof}
In order to have a more light notation, let $x_0:=S_0>0$ and $y_0:=I_0>0$. Also fix $\varepsilon:=1/2$. Define
$$T:=\begin{cases}
T_{\varepsilon,\beta,\gamma} &\text{ as in Lemma \ref{l1} if $\gamma\geq \beta$}\\
T_{x_0,y_0,\beta,\gamma} &\text{ as in \eqref{Test} if $\gamma< \beta$}.
\end{cases}$$
Note that, if $\gamma<\beta$ then $T$ meets the hypothesis of Lemma \ref{l2}.

Consider $Q_{x_0,y_0}$ as in Lemma \ref{l1} if $\gamma \geq \beta$ and $Q_{x_0,y_0}$ as in Lemma \ref{l2} if $\gamma <\beta$. 
We first prove that \eqref{modello3} admits a positive solution in $[0,T]$ -- note that $T$ is independent from $x_0,y_0$ when $\gamma\geq \beta$. To this end, consider the sequence of functions for $n\geq 1$
\begin{equation}\label{tx}
\tilde x_n(t):=\begin{cases}
x_0 &\text{if } {t\in [0,T/n]}\\
x_0 +\int_0^{t-T/n} f(\tilde x(s),\tilde y(s))\frac{(t-s)^{\alpha_1-1}}{\Gamma(\alpha_1)}ds&\text{if } {t\in (T/n,T]};
\end{cases}\end{equation}
\begin{equation}\label{ty}
\tilde y_n(t):=\begin{cases}
y_0 &\text{if } {t\in [0,T/n]}\\
y_0 -\int_0^{t-T/n}  f(\tilde x(s),\tilde y(s))\frac{(t-s)^{\alpha_2-1}}{\Gamma(\alpha_2)}ds&\text{if } {t\in (T/n,T]}.
\end{cases}
\end{equation}
Set for brevity $\tilde z_n(t):=(\tilde x_n(t),\tilde y_n(t))$. 

\emph{Claim 1.} The sequence of functions $\{\tilde z_n(t)\}_{n\geq 1}$ is well defined, equicontinuous and equibounded in $[0,T]$. In particular $\tilde z_n(t)\in Q_{x_0,y_0}$ for all $t\in[0,T]$ and for all $n\geq 1$. \\
\medskip
We prove the claim by showing by induction that for all $n$ and for all $k=1,\dots,n$, $\tilde z_n(t)$ is well defined in $t\in [0,kT/n]$, $\tilde z_n(t)\in Q_{x_0,y_0}$ for all $t\in [0,kT/n]$, and that 
\begin{equation}\label{equi}|z_n(t_1)-z_n(t_2)|\leq 2 \parallel f\parallel_{L^\infty(Q_{x_0,y_0})} \omega(|t_1-t_2|) \qquad \forall t_1,t_2\in[0,kT/n]\end{equation}
where $\omega$ is the modulus of continuity of the function $G(\cdot;\alpha_1,\alpha_2)$. 
If $k=1$, then by the definitions in \eqref{tx} and \eqref{ty} it follows that $z_n(t)=(x_0,y_0)$ for all $t\in[0,T/n]$ and the base of the induction readily follows. 
 Fix $k$ such that $1<k<n$ and assume  that $\tilde z_n(t)$ is well defined on $[0, k T/n]$,  $\tilde z_n(t)\in Q_{x_0,y_0}$ for every $t\in [0,kT/n]$, and that \eqref{equi} holds.
  Then using the second lines of \eqref{tx} and \eqref{ty} we have that the definition of $\tilde z_n(t)$ continuously extends to $[0, (k+1)T/n]$ by setting for $t\in (kT/n,(k+1)T/n]$
$$\tilde z_n(t):=\left(x_0+\int_0^{t-T/n} f(\tilde z_n(s))\frac{(t-s)^{\alpha_1}}{\Gamma(\alpha_1+1)}ds,y_0-\int_0^{t-T/n} f(\tilde z_n(s))\frac{(t-s)^{\alpha_2}}{\Gamma(\alpha_2+1)}ds\right),$$
By  applying Lemma \ref{l1} or Lemma \ref{l2} (according to the cases $\gamma\geq \beta$ and $\gamma<\beta$, respectively)
to $(x(s),y(s))=z(s):=\tilde z_n(s+(k-1)T/n)$  (which satisfies $(x(s),y(s)) \in Q_{x_0,y_0}$ for all $s\in [0,T/n]$ by inductive hypothesis) we have that $\tilde z_n(t)\in Q_{x_0,y_0}$ for all $t\in [k T,(k+1)T/n]$. 
It is left to show that
\begin{equation}\label{equi+}|z_n(t_1)-z_n(t_2)|\leq 2 \parallel f\parallel_{L^\infty(Q_{x_0,y_0})} \omega(|t_1-t_2|) \qquad \forall t_1,t_2\in[0,(k+1)T/n]\end{equation}
For all $t_1,t_2\in[T/n,(k+1)T/n]$, since $\tilde z_n(s)\in Q_{x_0,y_0}$ for all $s\in[0,(k+1)T/n]$, one has
\begin{align*}
|\tilde z_n(t_2)-\tilde z_n(t_1)|&\leq \int_{t_1-T/n}^{t_2-T/n}|f(\tilde z_n(s))|\frac{(t-s)^{\alpha_1}}{\Gamma(\alpha_1+1)}ds+\int_{t_1-T/n}^{t_2-T/n}|\tilde f(z_n(s))|\frac{(t-s)^{\alpha_2}}{\Gamma(\alpha_2+1)}ds\\
&\leq \parallel f\parallel_{L^\infty(Q_{x_0,y_0})}  \int_{t_1-T/n}^{t_2-T/n}\frac{(t-s)^{\alpha_1}}{\Gamma(\alpha_1+1)}+ \frac{(t-s)^{\alpha_2}}{\Gamma(\alpha_2+1)}ds\\
&\leq 2\parallel f\parallel_{L^\infty(Q_{x_0,y_0})} |G(t_2-T/n;\alpha_1,\alpha_2)-G(t_1-T/n;\alpha_1,\alpha_2)|\\
&\leq 2\parallel f\parallel_{L^\infty(Q_{x_0,y_0})}\omega(|t_2-t_1|).\end{align*}
The case in which $t_1,t_2\in[0,T/n]$ is trivial, because $\tilde z_n$ is constant in that interval. It is left to discuss the case in which $t_1\in[0,T/n]$ and $t_2\in[T/n,(k+1)T/n]$. In this case $|t_2-t_1|=t_2-t_1\geq t_2-T/n=|t_2-T/n|$. Then arguing as above one gets
\begin{align*}
|\tilde z_n(t_2)-\tilde z_n(t_1)|&=|\tilde z_n(t_2)-(x_0,y_0)| \leq \int_{0}^{t_2-T/n}|f(\tilde z_n(s))|\frac{(t-s)^{\alpha_1}}{\Gamma(\alpha_1+1)}ds\\
&\leq \int_{0}^{t_2-T/n}|f(\tilde z_n(s))|\frac{(t-s)^{\alpha_1}}{\Gamma(\alpha_1+1)}ds\\
&\leq 2\parallel f\parallel_{L^\infty(Q_{x_0,y_0})}\omega(|t_2-T/n|)\\
&\leq 2\parallel f\parallel_{L^\infty(Q_{x_0,y_0})}\omega(|t_2-t_1|)
.\end{align*}
This concludes the proof of the inductive step and, consequently, of the Claim 1. 

\medskip
Now, by Claim 1 and by Ascoli-Arzela's theorem, there exists a subsequence $\{\tilde z_{n_k}(t)\}$ converging uniformly in $[0,T]$ to a continuous limit function $ z(t)=(S(t),I(t))$ satisfying 
$z(t)\in  {Q_{x_0,y_0}}$ -- recall indeed that $Q_{x_0,y_0}\subset(0,+\infty)\times(0,+\infty)$ is a compact set.  This implies in particular $S(t)>0$ and $I(t)>0$ for all $t\in[0,T]$.  Since $f$ is continuous, then
$$f(\tilde z_{n_k}(t))\to f(z(t))=f(S(t),I(t))\qquad \text{as $k\to+\infty$},\quad\forall t\in[0,T]$$
moreover,
$$|f(\tilde z_{n_k}(t))|\leq \parallel f\parallel_{L^\infty({Q_{x_0,y_0}})}\quad \forall t\in[0,T].$$
Then, by Lebesgue's dominated convergence theorem,  for every fixed $t\in(0,T]$ and for $i=1,2$
$$\lim_{k\to\infty}\int_0^t f(\tilde z_{n_k}(s))\frac{(t-s)^{\alpha_i-1}}{\Gamma(\alpha_i)}ds= \int_0^t f( z(s))\frac{(t-s)^{\alpha_i-1}}{\Gamma(\alpha_i)}ds= \int_0^t f( S(s),I(s))\frac{(t-s)^{\alpha_i-1}}{\Gamma(\alpha_i)}ds$$
Therefore, for all $t\in(0,T]$, choosing $k$ sufficiently large to have $t>T/n_k$, one has $\tilde z_{n_k}(t)=(\tilde x_{n_k}(t),\tilde y_{n_k}(t))$ where $\tilde x_{n_k}$ and $\tilde y_{n_k}$ satisfy
\begin{align*}\lim_{k\to+\infty}\tilde x_{n_k}&=\lim_{k\to+\infty}x_0+\int_0^{t} f(\tilde z_{n_k}(s))\frac{(t-s)^{\alpha_1}}{\Gamma(\alpha_1+1)}ds-\int^t_{t-T/n_k} f(\tilde z_{n_k}(s))\frac{(t-s)^{\alpha_1}}{\Gamma(\alpha_1+1)}ds
\\&
= x_0+\int_0^{t} f(S(s),I(s))\frac{(t-s)^{\alpha_1-1}}{\Gamma(\alpha_1)}ds
\end{align*}
and
\begin{align*}\lim_{k\to+\infty}\tilde y_{n_k}&=\lim_{k\to+\infty} y_0-\int_0^{t} f(\tilde z_{n_k}(s))\frac{(t-s)^{\alpha_2-1}}{\Gamma(\alpha_2)}ds-\int^t_{t-T/n_k} f(\tilde z_{n_k}(s))\frac{(t-s)^{\alpha_2-1}}{\Gamma(\alpha_2)}ds
\\&= y_0-\int_0^{t} f(S(s),I(s))\frac{(t-s)^{\alpha_2-1}}{\Gamma(\alpha_2)}ds.
\end{align*}
On the other hand $(\tilde x_{n_k}(t),\tilde y_{n_k}(t))\to (S(t),I(t))$ as $k\to \infty$ for all $t\in(0,T]$, and recalling $x_0=S_0$ and $y_0=I_0$ one deduces that 
$$S(t)=S_0+\int_0^{t} f(S(s),I(s))\frac{(t-s)^{\alpha_1-1}}{\Gamma(\alpha_1)}ds$$
and
$$I(t)=I_0-\int_0^{t} f(S(s),I(s))\frac{(t-s)^{\alpha_2-1}}{\Gamma(\alpha_2)}ds.$$
Then $(S(t),I(t))$ is the required solution of \eqref{modello3} in $(0,T]$. If $\beta>\gamma$ then we are done, because we proved the local existence of a solution of \eqref{modello3}. If $\beta\leq\gamma$, then we can iteratively extend  $(S(t),I(t))$ to $(0,+\infty)$. For instance by applying the result to the initial datum $(x_0,y_0):=(S(T/2),I(T/2))$ we then obtain a solution defined in $[0,T/2+T]$ and so on.

\end{proof}

\subsection{Numerical simulations}\label{sec:CCnum}

{The numerical discretization of system \eqref{modello3} follows \cite{giga2019AA}. Let us consider a general equation
\begin{equation}\label{eq:num}
	{^C D}^{\alpha}_tu(t) = f(u(t)).
\end{equation}
and consider a numerical grid which uniformly divides the time interval $[0,T]$ into $\nt$ steps of length $\dt$. We denote by $u^{n}=u(t^{n})$, with $t^{n}=n\dt$.
Let $\alpha\in(0,1)$, then the Caputo derivative can be approximated as 
\begin{equation*}
	{^C D}^{\alpha}_tu^{n}=\frac{1}{\Gamma(2-\alpha)\dt^{\alpha}}\Big(u^{n}-\sum_{j=0}^{n-1}C_{n,j}u^{j} \Big),
\end{equation*}
with 
\begin{align*}
	C_{n,0} = g(n),\qquad C_{n,j} = g(n-j)-g(n-(j-1)) \quad\text{ for $j=1,\dots,n-1$}
\end{align*}
and $g(r) = r^{1-\alpha}-(r-1)^{1-\alpha} \text{ for $r\geq1$}$. The numerical scheme to solve \eqref{eq:num} is then given by
\begin{equation*}\label{giga}
	u^{n+1} = \sum_{j=0}^{n-1}C_{n,j}u^{j}+\Gamma(2-\alpha)\dt^{\alpha}f(u^{n}).
\end{equation*} 
Let us denote by $S^{n}=S(t^{n})$ and $I^{n}=I(t^{n})$. By applying this discretization to system \eqref{modello3} we obtain
\begin{align*}
	S^{n+1} &= \sum_{j=0}^{n-1}C_{n,j}S^{j}+\Gamma(2-\alpha_{1})\dt^{\alpha_{1}}f_{S}(S^{n},I^{n})
\\
	I^{n+1} &= \sum_{j=0}^{n-1}C_{n,j}S^{j}+\Gamma(2-\alpha_{2})\dt^{\alpha_{2}}f_{I}(S^{n},I^{n}),
\end{align*}
with
\begin{align*}
	f_{S}(S,I) = -\beta\frac{S}{S+I}I+\gamma I \qquad\text{and}\qquad f_{I}(S,I) = \beta\frac{S}{S+I}I-\gamma I.
\end{align*}

 We show the evolution in time of $S(t)$, $I(t)$ and their sum $S(t)+I(t)$ as $\alpha_{1}, \alpha_{2}$ changes: in Figures \ref{fig:dato1} we considered a case in which $\beta>\gamma$ and in  \ref{fig:dato5} $\gamma<\beta$. Note  in plots (a) and (b) that the steepness of the solutions $S(t)$ and $I(t)$ in the long run are related to the size of $\alpha_1$ and $\alpha_2$, respectively: this can be interpreted as a time delay effect of the Caputo fractional operator. 
An interesting phenomenon is also the lack of monotonicity of the solutions, in contrast with ordinary SIS models. 

%
  The sum of the two classes, see plots (c), shows that $N(t)=S(t)+I(t)$ is in general not monotone. For instance, in Figure \ref{fig:dato1}, $N(t)$ first decreases and then increases when $\alpha_{1}>\alpha_{2}$, while it first increases and then decreases when $\alpha_{1}<\alpha_{2}$ (a symmetrical behavior emerges in Figure \ref{fig:dato5}).  In case $\alpha_{1}=\alpha_{2}$ then the sum is constant and we recover the theory developed in \cite{balzotti2020FF, hassouna2018CSF}. A rigorous study of the symmetries emerging in the simulations and, in particular the intersections of all the functions $N(t)$ at the same time, is still under investigation. 
}

\begin{figure}[h!]
\centering
	\subfloat[][$S(t)$]{
	\includegraphics[width=0.31\columnwidth]{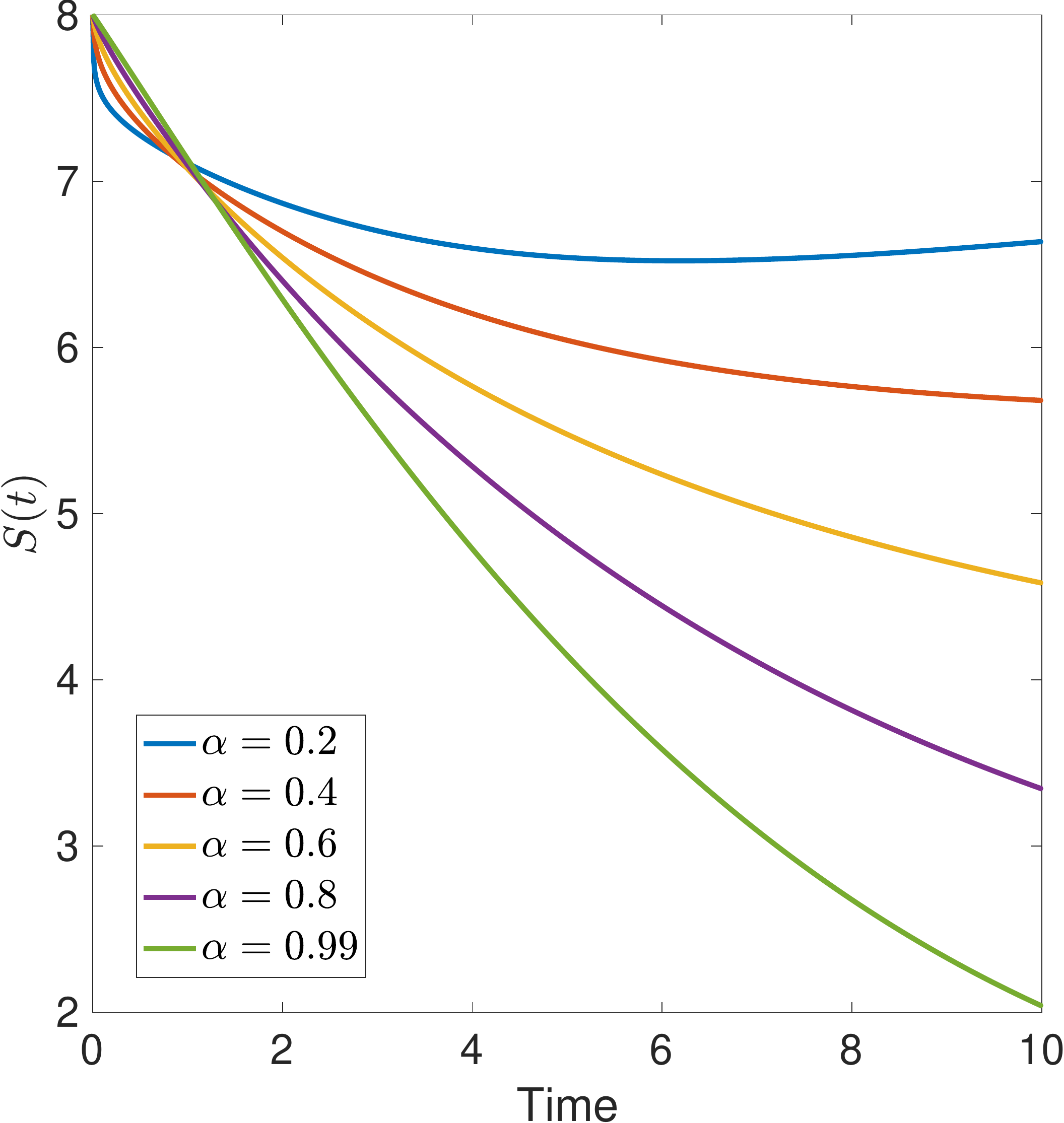}
	}\,
	\subfloat[][$I(t)$]{
	\includegraphics[width=0.31\columnwidth]{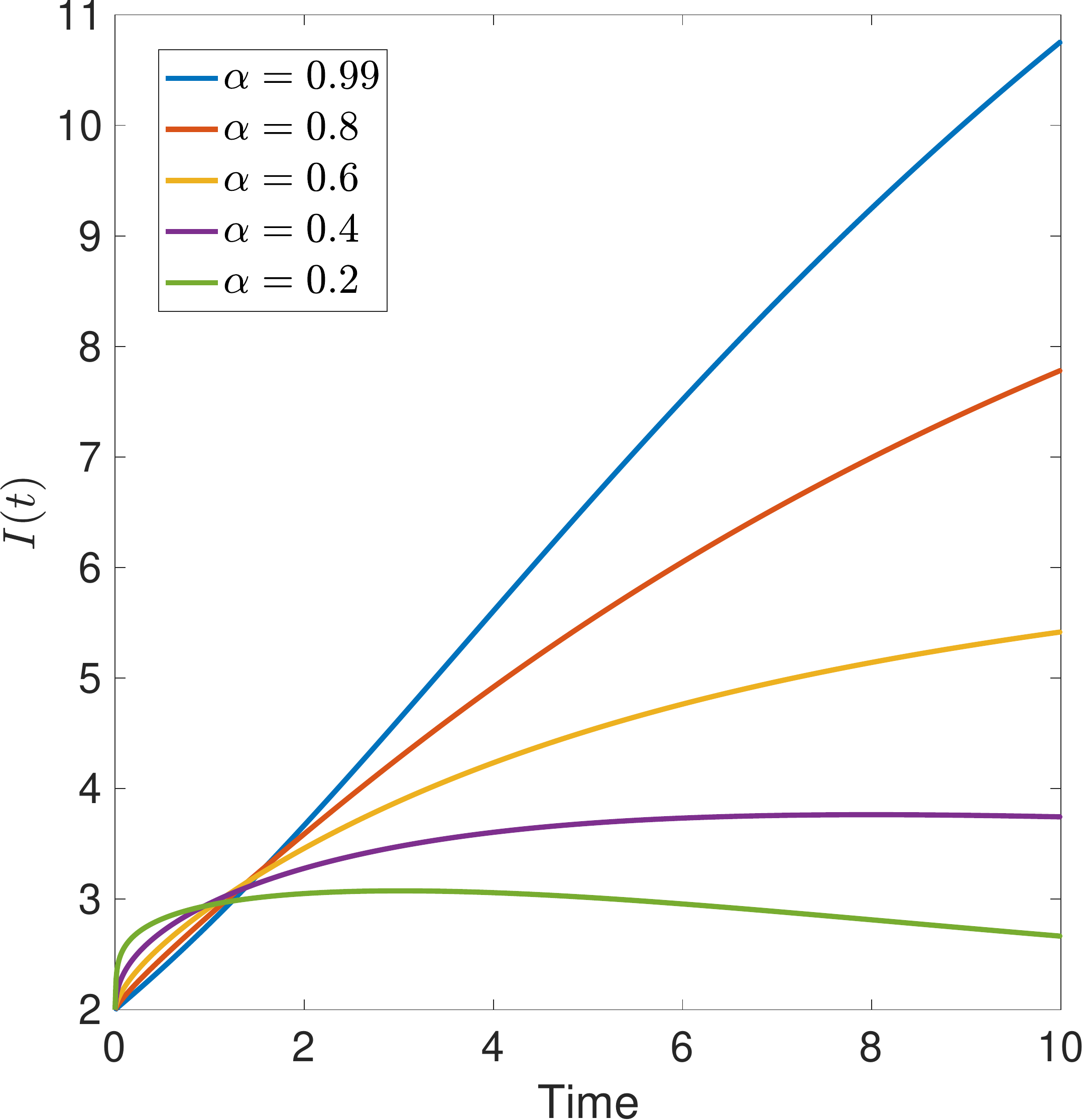}
	}\,
	\subfloat[][$S(t)+I(t)$]{
	\includegraphics[width=0.31\columnwidth]{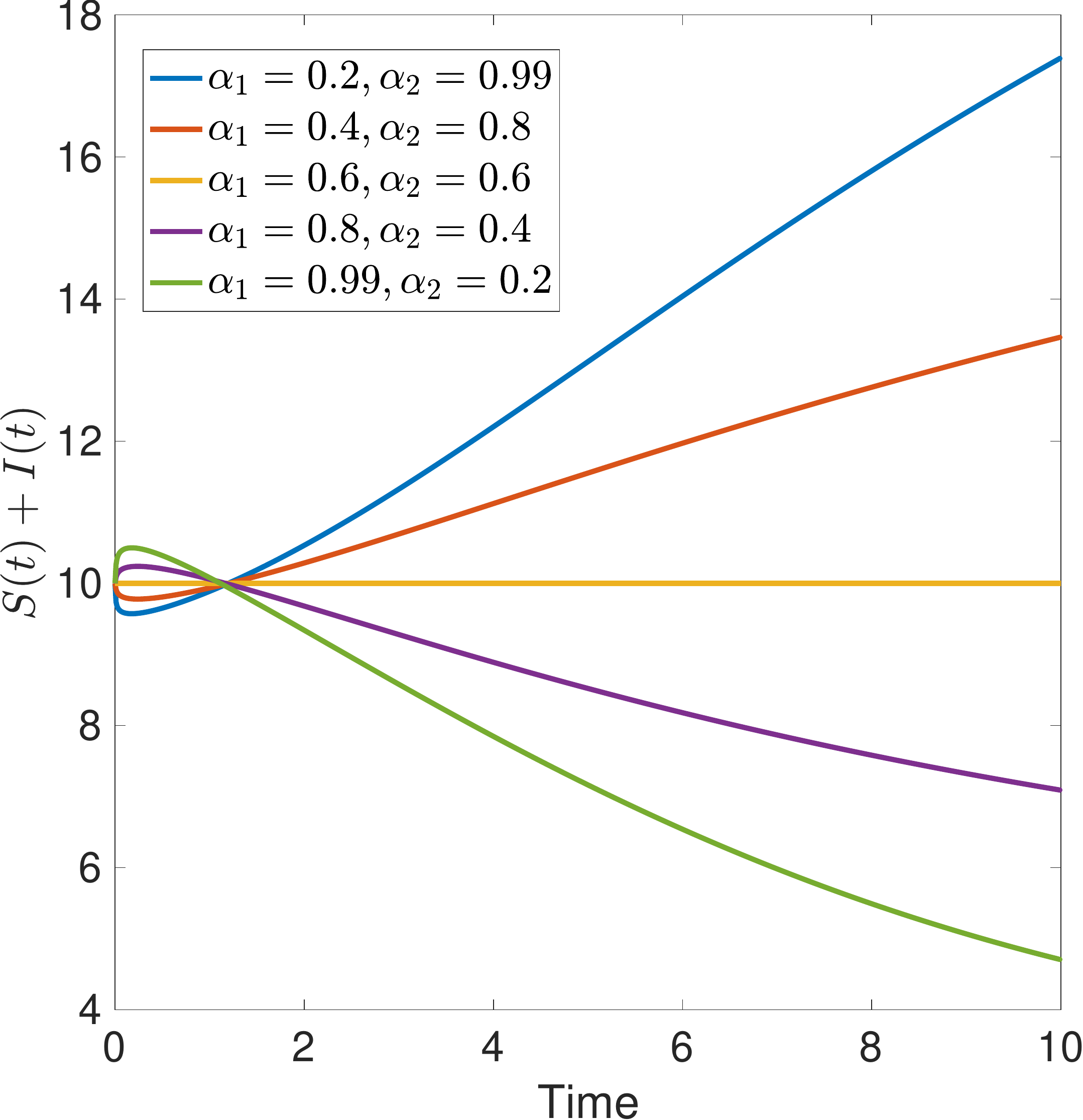}
	}		
\caption{Numerical solutions to \eqref{modello3} with $\beta=0.7$, $\gamma=0.2$, $I_{0}=2$ and $S_{0}=8$.}
\label{fig:dato1}
\end{figure}

%
%

\begin{figure}[h!]
\centering
	\subfloat[][$S(t)$]{
	\includegraphics[width=0.31\columnwidth]{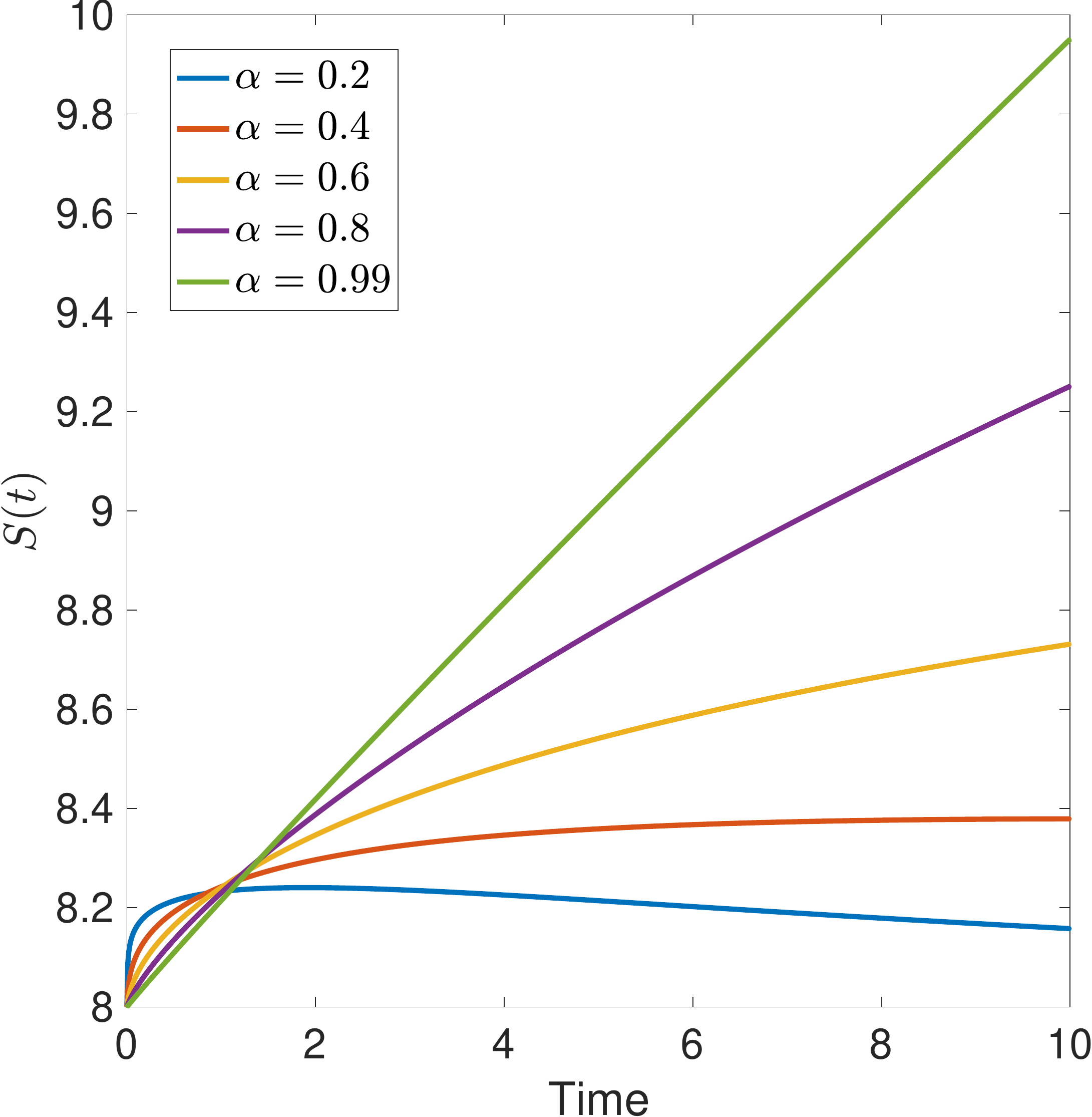}
	}\,
	\subfloat[][$I(t)$]{
	\includegraphics[width=0.31\columnwidth]{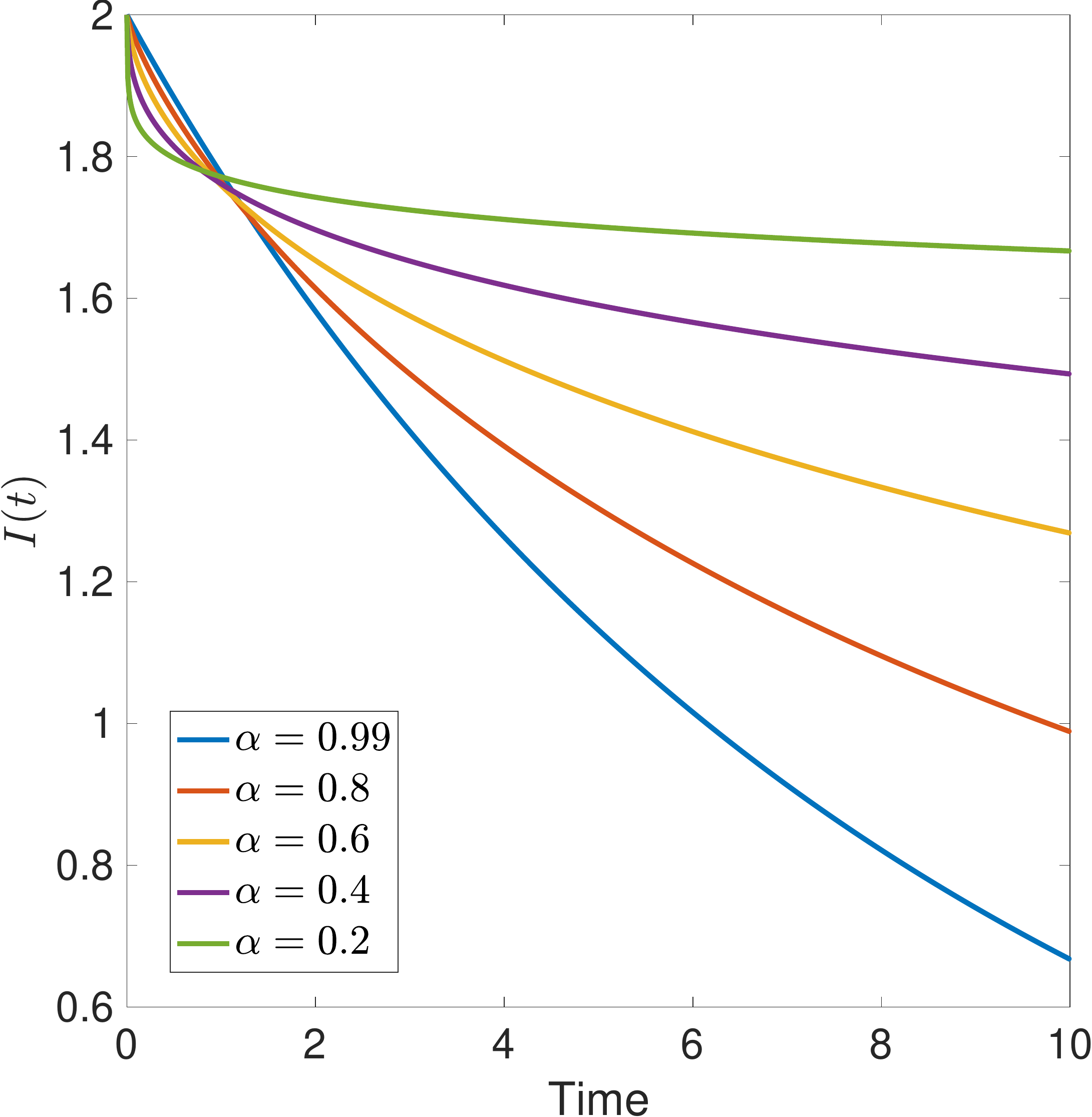}
	}\,
	\subfloat[][$S(t)+I(t)$]{
	\includegraphics[width=0.31\columnwidth]{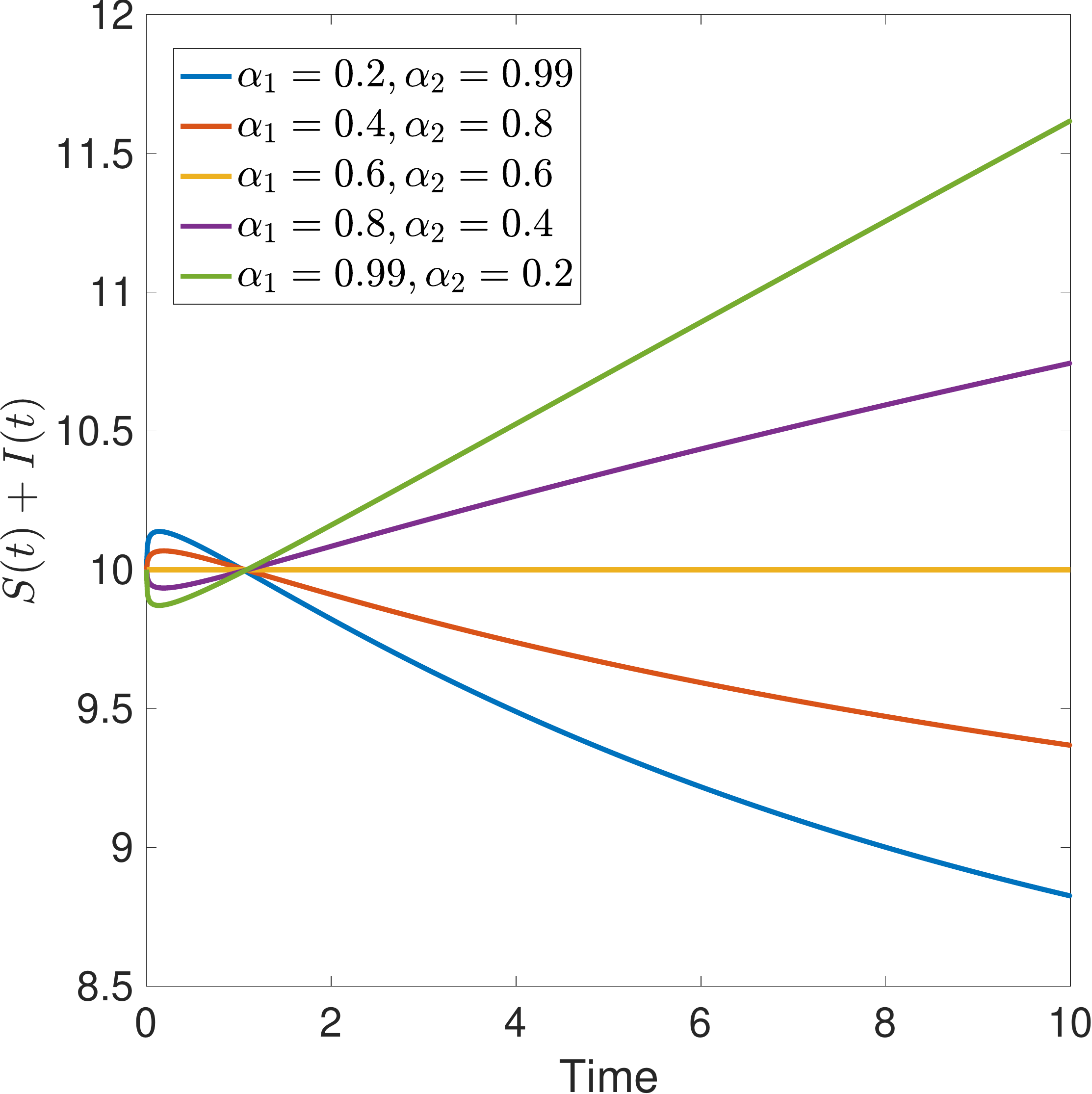}
	}		
\caption{Numerical solutions to \eqref{modello3} with $\beta=0.1$, $\gamma=0.2$, $I_{0}=8$ and $S_{0}=2$.}
\label{fig:dato5}
\end{figure}

\section{Caputo-Fabrizio fractional epidemic models}\label{s3}
In this section we are concerned with the SIS model using Caputo-Fabrizio type fractional derivatives. We refer to \cite{higazy2020AEJ, khan2021AIMS, moore2019ADE, ullah2020DCDSS} for some examples of epidemic models based on the Caputo-Fabrizio fractional operator.

More precisely, here we study the  the fractionary SIS system \eqref{SIS-CF}, which we recall to be
\begin{equation}\label{SIS}
\begin{cases}
^{\text{CF}}D^\alpha_t S= -\left(\frac{\beta}{S+I} S-\gamma\right)I\\
I'= \left(\frac{\beta}{S+I} S-\gamma\right)I\\
S(0)=S_0\\
I(0)=I_0.
\end{cases}
\end{equation}
where $\beta,\gamma>0$, $S_0,I_0\geq 0$, $I_0+S_0>0$. Also recall that the Caputo-Fabrizio operator of order $\alpha\in[0,1]$ for a function $u\in H^1((a,b))$, $a<b$ is 
$$^{\text{CF}}D^\alpha_t u(t):=\frac{M(\alpha)}{1-\alpha}\int_a^t u'(\tau) e^{-\frac{\alpha}{1-\alpha}(t-\tau)}d\tau.$$
where $M(\alpha)$ is a non-negative scaling factor satisfying $M(0)=M(1)=1$.
The indentity
\begin{equation}\label{dD}
\frac{d}{dt} \,^{\text{CF}}D^\alpha_t u(t)=\frac{M(\alpha)}{1-\alpha}u'(t)-\frac{\alpha}{1-\alpha}\,^{\text{CF}}D^\alpha_tu(t)\quad \forall t\in(a,b),\alpha\in[0,1).
\end{equation}
is crucial to reconduct \eqref{SIS} to a system of ordinary differential equations, which is the first step of our investigation. 
\medskip
We make the following key assumption, relating the order of derivation with the system parameters:
\begin{align}\label{alphagamma}
\gamma<\frac{\alpha}{1-\alpha}, \\
M(\alpha)\geq \alpha \quad \forall \alpha\in[0,1].\label{alphaMalpha}
\end{align}
Note that, for any $\gamma>0$ the  condition \eqref{alphagamma} is satisfied by choosing $\alpha$ sufficiently close to $1$ whereas \eqref{alphaMalpha} trivially holds choosing  $M(\alpha)\equiv 1$, which is a setting earlier explored and motivated in \cite{nieto}. 
\subsection{Solutions and equilibria}\label{s31}
As anticipated above, we begin by rewriting   \eqref{SIS} as a system of ordinary equations for $\alpha\in[0,1)$. 
\begin{Theorem}\label{charcf} 
Let  $\beta,\gamma>0$, $S_0,I_0\geq 0$, $I_0+S_0>0$ and 
assume the conditions \eqref{alphagamma} and \eqref{alphaMalpha}. Let 
$$B_\alpha:=\frac{1}{2}(\alpha + M(\alpha)+  (1 - \alpha)( \beta- \gamma)),\qquad C_\alpha:=M(\alpha)(\alpha-(1-\alpha)\gamma)$$
and
\begin{align*}
P_\alpha(S_0,I_0)&:=M(\alpha)S_0+\alpha I_0+(1-\alpha)\left(\beta-\gamma-\frac{\beta}{S_0+I_0} I_0\right)I_0.
\end{align*}
Also define the  function $g_\alpha:\RR\to\RR$
 $$g_\alpha(x):=\frac{-B_\alpha x+P_\alpha(S_0,I_0)/2 + \sqrt{(B_\alpha x-P_\alpha(S_0,I_0)/2 )^2-
	C_\alpha x^2+M(\alpha)P_\alpha(S_0,I_0) x)}}{ M(\alpha)}.$$

Then any couple non-negative absolutely functions continuous $(S(t),I(t))$ is a solution of fractional system \eqref{SIS} if and only if 
\begin{equation}\label{algebraic2}
\begin{split}S(t)&=g_\alpha(I(t))\qquad \forall t\geq 0
\end{split}
\end{equation}
and
$I(t)$ solves 
\begin{equation}\label{sisI}
\begin{cases}
I'=\bigg(\beta-\gamma-\dfrac{\beta I}{g_\alpha(I)+I}\bigg) I\\
I(0)=I_0.
\end{cases}
\end{equation}
\end{Theorem}
\begin{proof}

  Preliminarly remark that, by \eqref{dD} and by the first equation of \eqref{SIS}
\begin{align*}
\frac{d}{dt}D^{\text{CF}}_\alpha S&=\frac{M(\alpha)}{1-\alpha}S'-\frac{\alpha}{1-\alpha}D^{CF}_\alpha S\\
&=\frac{M(\alpha)}{1-\alpha}S'+ \frac{\alpha}{1-\alpha}\left(\frac{\beta}{S+I} S-\gamma\right)I.
\end{align*}
By differentiating both sides of the first equation of \eqref{SIS}, we deduce that \eqref{SIS} is equivalent to the ordinary system
\begin{equation}\label{SISordinary}
\begin{cases}
\frac{M(\alpha)}{1-\alpha}S'+ \frac{\alpha}{1-\alpha}\left(\frac{\beta}{S+I} S-\gamma\right)I= -\left(\left(\frac{\beta}{S+I} S-\gamma\right)I\right)'\\
I'= \left(\frac{\beta}{S+I} S-\gamma\right)I\\
S(0)=S_0\\
I(0)=I_0.
\end{cases}
\end{equation}

%
Now define
\begin{align*}
P(t):=&M(\alpha)S(t)+\alpha I(t)+ (1-\alpha)\left(\frac{\beta}{S(t)+I(t)} S(t)-\gamma\right)I(t)\\
=&M(\alpha)S(t)+\alpha I(t)+ (1-\alpha)\left(\beta-\gamma-\frac{\beta}{S(t)+I(t)} I(t)\right)I(t)
\quad \forall t\geq 0
\end{align*}
and remark that \eqref{SISordinary} implies
$P'(t)=0$ and consequently $P(t)=P(0)=P_\alpha(S_0,I_0)$ for all $t\geq0$. 
Then, if  $I(t)$ and $S(t)$ are solutions of \eqref{SISordinary} (or equivalently, of \eqref{SIS}) then 
\begin{equation}\label{algebraic}
M(\alpha)S(t)+\alpha I(t)+ (1-\alpha)\left(\beta-\gamma-\frac{\beta}{S(t)+I(t)} I(t)\right)I(t)=P_\alpha(S_0,I_0).\end{equation}
Solving above equation with respect to $S(t)$
%
and selecting the only possibly non-negative solution, we obtain for all $t\geq0$
\begin{equation*}
\begin{split}S(t)=&\frac{1}{ M(\alpha)}\left(-B_\alpha I(t)+P_\alpha(S_0,I_0)/2 \right.\\
&\left.+ \sqrt{(B_\alpha I(t)-P_\alpha(S_0,I_0)/2 )^2-
	C_\alpha I^2(t)+M(\alpha)P_\alpha(S_0,I_0) I(t)}\right)\\
	=&g_\alpha(I(t)).
\end{split}
\end{equation*}
Incidentally notice that if $\alpha=1$ then  $P_1(I_0)=I_0+S_0=:N$, $B_1=C_1=1$ we recover from above relation the classical identity $S(t)=N-I(t)$.
We check that $S(t)$, as a function $g_\alpha$ of $I(t)$, is well defined. To this end set note that  $g_\alpha(x)$ is defined in $\RR$, because 
\begin{align*}&(B_\alpha x-P_\alpha(S_0,I_0)/2 )^2-
	C_\alpha x^2+M(\alpha)P_\alpha(S_0,I_0) x=\\
	&(B_\alpha^2 -C_\alpha)x^2-(B_\alpha-M_\alpha)P_\alpha(S_0,I_0)x+P_\alpha^2(S_0,I_0)/4\geq 0
	\end{align*}
for all $x\in\RR$. More precisely, above inequality holds because the discriminant of above polynomial reads 
$$(B_\alpha-M_\alpha)^2P_\alpha^2(S_0,I_0)-(B_\alpha^2 -C_\alpha)P_\alpha^2(S_0,I_0)=-(1-\alpha)\beta M(\alpha)P_\alpha^2(S_0,I_0)<0$$
for all $\alpha\in[0,1)$, whereas we remarked above that if $\alpha=1$ then $g_\alpha(x)=N-x$.

Plugging the identity $S(t)=g_\alpha(I(t))$ in the second equation of \eqref{SIS} we obtain \eqref{SISordinary}.

Finally we check that $S(t)$ and $I(t)$ are non-negative. By \eqref{algebraic2}, $S(t)\geq 0$ if and only if $I(t)\in [0,P_\alpha(S_0,I_0)/(\alpha-(1-\alpha)\gamma)]$ for all $t\geq 0$. 
  Remark that this condition is satisfied for $t=0$, indeed we have
	\begin{align*}
	P_\alpha(S_0,I_0)&\geq \alpha I_0+(1-\alpha)(\beta-\gamma-\frac{\beta}{S_0+I_0}I_0)I_0\\
	&\geq (\alpha -(1-\alpha)\gamma) I_0.
	\end{align*}
Moreover if $I(t)=P_\alpha(S_0,I_0)/(\alpha-(1-\alpha)\gamma)$ then $I'(t)=-\gamma I(t)<0$ and, consequently $I(t)\leq P_\alpha(S_0,I_0)/(\alpha-(1-\alpha)\gamma)$. Finally, remark that $0$ is an equilibrium for \eqref{SISordinary} and that the velocity field $f_\alpha(I)=\bigg(\beta-\gamma-\dfrac{\beta I}{g_\alpha(I)+I}\bigg) I$ is locally Lipschitz continuous. Therefore, by the local uniqueness of the solutions of \eqref{SISordinary} if $I_0>0$ then $I(t)>0$ for all $t$. 
\end{proof}

%
\subsubsection{Equilibria}
We now characterize the equilibria and we study the asymptotic behavior of \eqref{SIS}.
\begin{Proposition}\label{pequilibria}
Assume conditions \eqref{alphagamma} and \eqref{alphaMalpha} and let $S_0,I_0\geq 0$. The equilibria of the system \eqref{sisI} are $0$ and 
$$E_\alpha(S_0,I_0):=\frac{(\beta - \gamma) P_\alpha(S_0,I_0)}{\alpha(\beta-\gamma)+M(\alpha)\gamma}.$$
Setting $R:=\beta/\gamma$, $E_\alpha(I_0)>0$ if and only if $R>1$. In this case
$$I(t)\to E_\alpha(S_0,I_0), \quad S(t)\to \frac{\gamma}{\beta-\gamma} E_\alpha(S_0,I_0)= \frac{\gamma P_\alpha(S_0,I_0)}{\alpha(\beta-\gamma)+M(\alpha)\gamma} \qquad \text{as }t\to+\infty.$$
Finally, if $E_\alpha(S_0,I_0)\leq0$, that is if $R\leq 1$, then 
$$I(t)\to 0, \quad S(t)\to P_\alpha(S_0,I_0) \qquad \text{as }t\to+\infty.$$
\end{Proposition}
\begin{proof}
	The first part of the claim follows by a direct computation, using in particular  \eqref{alphaMalpha}  and the fact that $\beta,\gamma>0$ implies $R>0$. Let $x_R=E_\alpha(S_0,I_0)$ if $R>1$ and $x_R=0$ if $R\leq 0$. We proved that in Theorem \ref{charcf} that $I(t)$ is non-negative, and consequently $[0,+\infty)$ is an invariant set for the dynamics \eqref{SISordinary}. Moreover, by a direct computation one can check that the function $V(x):=(x-x_R)^2$ is a Lyapunov function for \eqref{SISordinary} in $[0,+\infty)$ and, consequently, $x_R$ is a globally asymptocally stable equilibrium for \eqref{SISordinary} and this concludes the proof. \end{proof} 
	
Note that, as in the classical case, the qualitative properties of the system strongly depend on the reproduction number $R=\beta/\gamma$. The value $E_\alpha(S_0,I_0)$ can be viewed as a fractional generalization of the endemic equilibrium of classical SIS models, which is a stable equilibrium when $R>1$. 

\medskip
The next result deals with the monotonicity and the asymptotic behaviour of the function $N(t):=S(t)+I(t)$, which is not constant unless we are in the ordinary case $\alpha=1$. As we show below, the asymptotic behaviour of $N(t)$ can be used for inverse problems, i.e., the goal to reconstruct the fractional order $\alpha$ of \eqref{SIS} from the observed data on the long range. 

\begin{Proposition}\label{pN}
	Assume $S_0,I_0> 0$ and $\alpha\in[0,1)$. If $R>1$ and if $I_0\in(0,E(S_0,I_0))$ then $N(t):=I(t)+S(t)$ is a strictly increasing function converging to 
$$N_\alpha(S_0,I_0):=\frac{\beta}{\beta-\gamma}E_\alpha(S_0,I_0)=\frac{\beta P_\alpha(S_0,I_0)}{\alpha(\beta-\gamma)+M(\alpha)\gamma}$$
as $t\to+\infty$.
If otherwise either $R> 1$ and $I_0>E_\alpha(S_0,I_0)$ or $R\leq 1$,  then $N(t)$ is  strictly decreasing and tends to $ P_\alpha(S_0,I_0)/M(\alpha)$ as $t\to +\infty$.
\end{Proposition}
\begin{proof}
We preliminary remark that $g_\alpha(x)$ is a  convex function, indeed 
\begin{equation}\label{gconvex}
g_\alpha''(x)=\frac{P_\alpha^2(S_0,I_0)(1-\alpha)\beta}{4 M(\alpha)\sqrt{(B_\alpha x-P_\alpha(S_0,I_0)/2)^2-C_\alpha x^2+M(\alpha)P_\alpha(S_0,I_0)}}\geq 0 \quad \forall x\in\RR.
\end{equation}
Moreover, by a direct computation, 
\begin{equation*}
1+g'_\alpha(I_0)=1+\frac{2 (1-\alpha) \beta  M(\alpha) (I_0+S_0)^3 P_\alpha(S_0,I_0)^2}{\left| I_0 (I_0+S_0)^2+(1-\alpha)
   I_0^2 \beta\right| ^3}> 0\end{equation*}
 and this, together with the convexity of $g_\alpha$, implies 
\begin{equation}\label{ginitial}
1+g'_\alpha(x)>0 \quad \forall x\geq I_0.
\end{equation}
Also remark that, using the assumption \eqref{alphaMalpha}, i.e., $M(\alpha)\geq \alpha$, we obtain 
\begin{align*}
1+g'_\alpha(E_\alpha(S_0,I_0))=&1+ \frac{M(\alpha) \left((1-\alpha)(\beta - \gamma)\gamma -\alpha \beta \right)}{(1-\alpha) (\beta -\gamma )^2+\beta  M(\alpha)} \\
\geq& 1+ \frac{M(\alpha) \left((1-\alpha)(\beta - \gamma)\gamma - \beta \right)}{(1-\alpha) (\beta -\gamma )^2+\beta  M(\alpha)} \\
=& \frac{(1-\alpha)(\beta-\gamma)(M(\alpha)\gamma+\beta-\gamma)}{(1-\alpha) (\beta -\gamma )^2+\beta  M(\alpha)} 
   \end{align*}
therefore 
\begin{equation}\label{ge}
\beta>\gamma\quad \Rightarrow\quad1+g'_\alpha(x)>0 \quad \forall x\geq E_0(S_0,I_0) .
\end{equation}

We conclude this preliminary study on $g_\alpha$ by noticing that
\begin{equation*}
1+g'_\alpha(0)=(1-\alpha)(1-\beta+\gamma)\end{equation*}
therefore 
\begin{equation}\label{g0}
\beta\leq\gamma\quad \Rightarrow\quad1+g'_\alpha(x)>0 \quad \forall x\geq 0 .
\end{equation}

Now,  by Theorem \ref{charcf} 
\begin{equation}\label{N}N'(t)=I'(t)+S'(t)=I'(t)(1+g'_\alpha(I(t)).\end{equation}
Assume $R>1 \text{ and } I_0\in(0,E(S_0,I_0))$. Then 
\begin{equation}\label{rel}
I'(t)=\left(\beta-\gamma-\frac{\beta I(t)}{I(t)+g_\alpha(I(t))}\right)I(t)>0 \qquad \forall t\geq 0\end{equation}
In particular $I(t)\geq I_0$ for all $t\geq 0$ and, in view of \eqref{ginitial} and of \eqref{N}, we deduce $N'(t)>0$ for all $t>0$, hence $N(t)$ is strictly increasing.

 Assume now that either $R\leq 1 \text{ or } I_0\geq E(S_0,I_0)$.
 
  First we discuss the case in which $R>1$ and $I_0> E(S_0,I_0)$. Then $I'(t)<0$ and, since $E_\alpha(S_0,I_0)$ is a (globally asymptotically stable) equilibrium by Proposition \ref{pequilibria}, then $I(t)>E_\alpha(S_0,I_0)$ for all $t\geq 0$. Then, in view of \eqref{ge} and of \eqref{N}, we deduce $N'(t)<0$ for all $t\geq 0$.

Finally, if $R\leq 1$ then $I'(t)<0$. Since $I(t)> 0$ for all $t>0$ then, in view of \eqref{g0} and of \eqref{N}, we deduce $N'(t)<0$ for all $t\geq 0$ and this concludes the proof. \end{proof}

From above result, we can estimate the fractional order $\alpha$ from the asymptotic behaviour of the total population $N(t)$. Indeed, if $R>1$ and $N(t)\to N_\infty$ then the corresponding fractional order $\alpha$ is the solution of the equation
$$\frac{\beta P_\alpha(S_0,I_0)}{\alpha(\beta-\gamma)+M(\alpha)\gamma}=N_\infty.$$
Assuming as in \cite{nieto} $M(\alpha)\equiv 1$, and setting $N_0:=I_0+S_0$, above equation reduces to the explicit formula

$$\alpha=1-\frac{(N_\infty-N_0)N_0\beta}{\beta I_0 S_0 -\beta N_0 (I_0 + \gamma I_0 - N_\infty) - \gamma N_0 N_\infty}.$$
Note in particular that $\alpha=1$ if and only if $N_\infty=N_0$, confirming the fact that, in the proposed mixed fractional model, $N(t)$ is constant if and only if $\alpha=1$. 

{
\subsection{Numerical simulations}\label{s32}
The numerical discretization of \eqref{SIS} is based on the results of Theorem \ref{charcf}. Let us consider again the numerical grid introduced in Section \ref{sec:CCnum}. To compute the discrete evolution in time of $I$ and $S$ we first solve system \eqref{sisI} by a proper ODE solver and then we discretize \eqref{algebraic2}. Specifically, we use the MATLAB tool \texttt{ode23t} to compute $I^{n+1}$ and then, following \eqref{algebraic2}, we obtain $S^{n+1}=g_{\alpha}(I^{n+1})$.

\begin{figure}[h!]
\subfloat[][$S(t)$]{
\includegraphics[width=0.3\columnwidth]{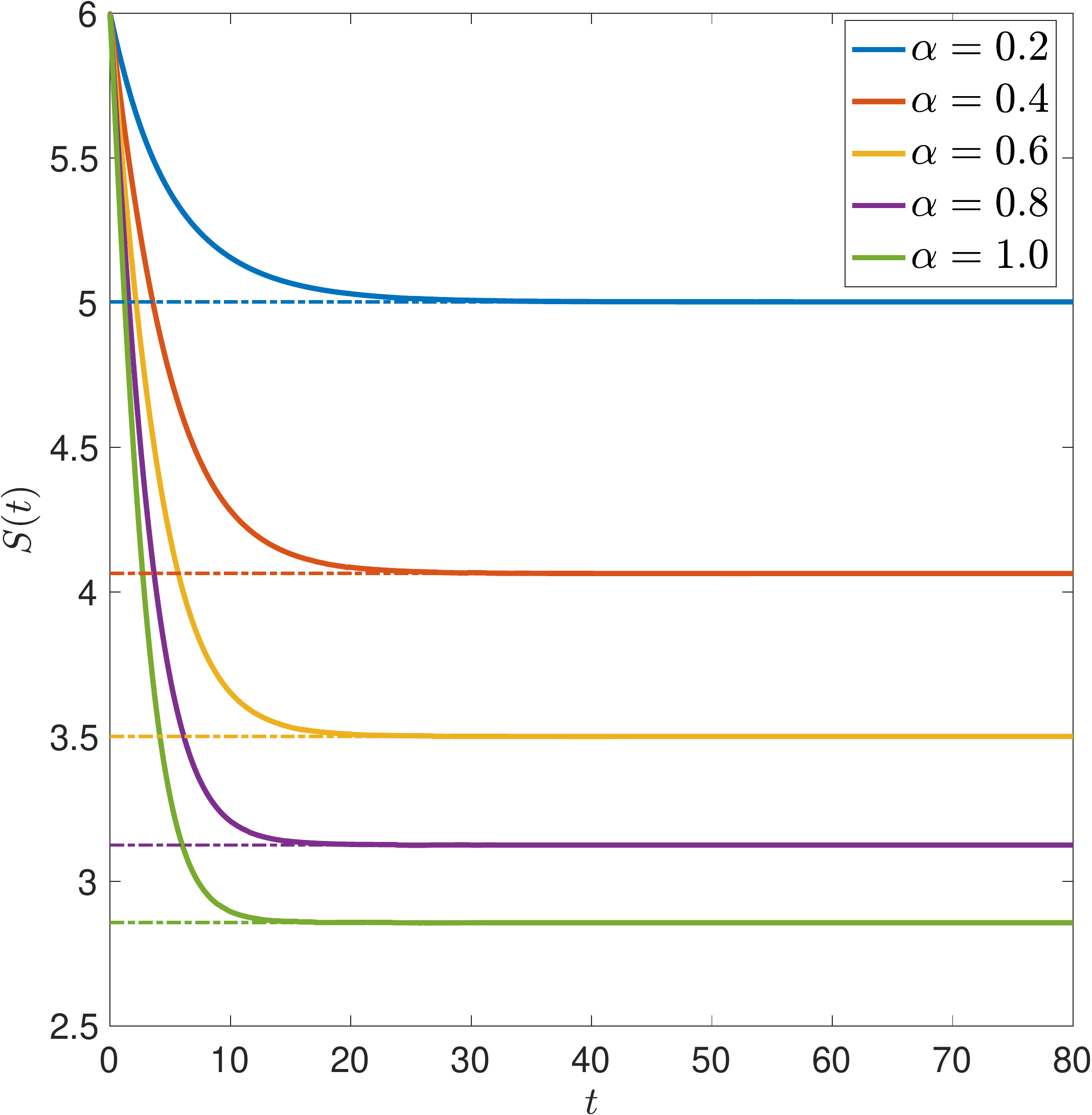}
}
\quad
\subfloat[][$I(t)$]{
\includegraphics[width=0.3\columnwidth]{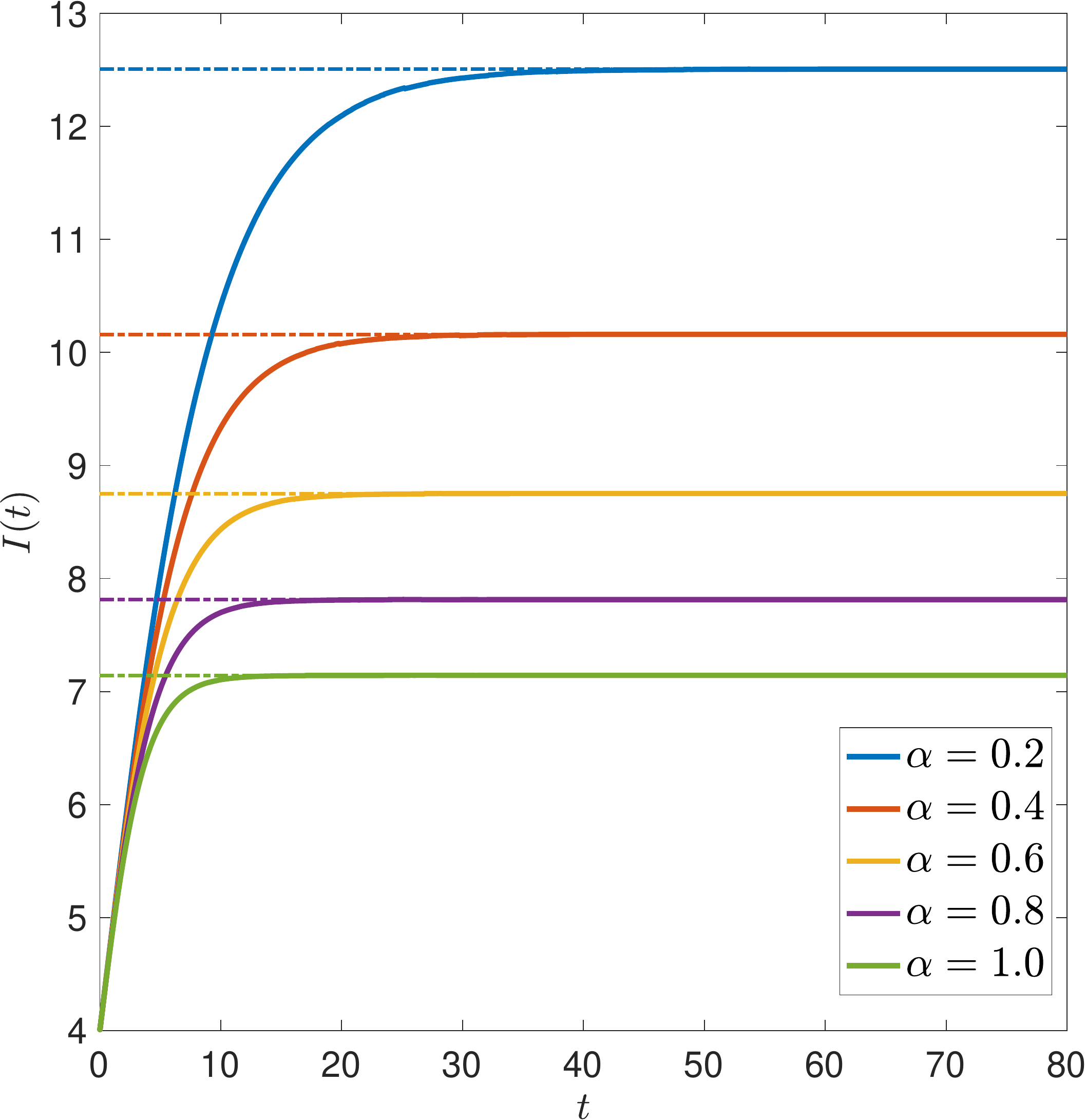}
}
\quad
\subfloat[][$S(t)+I(t)$]{
\includegraphics[width=0.3\columnwidth]{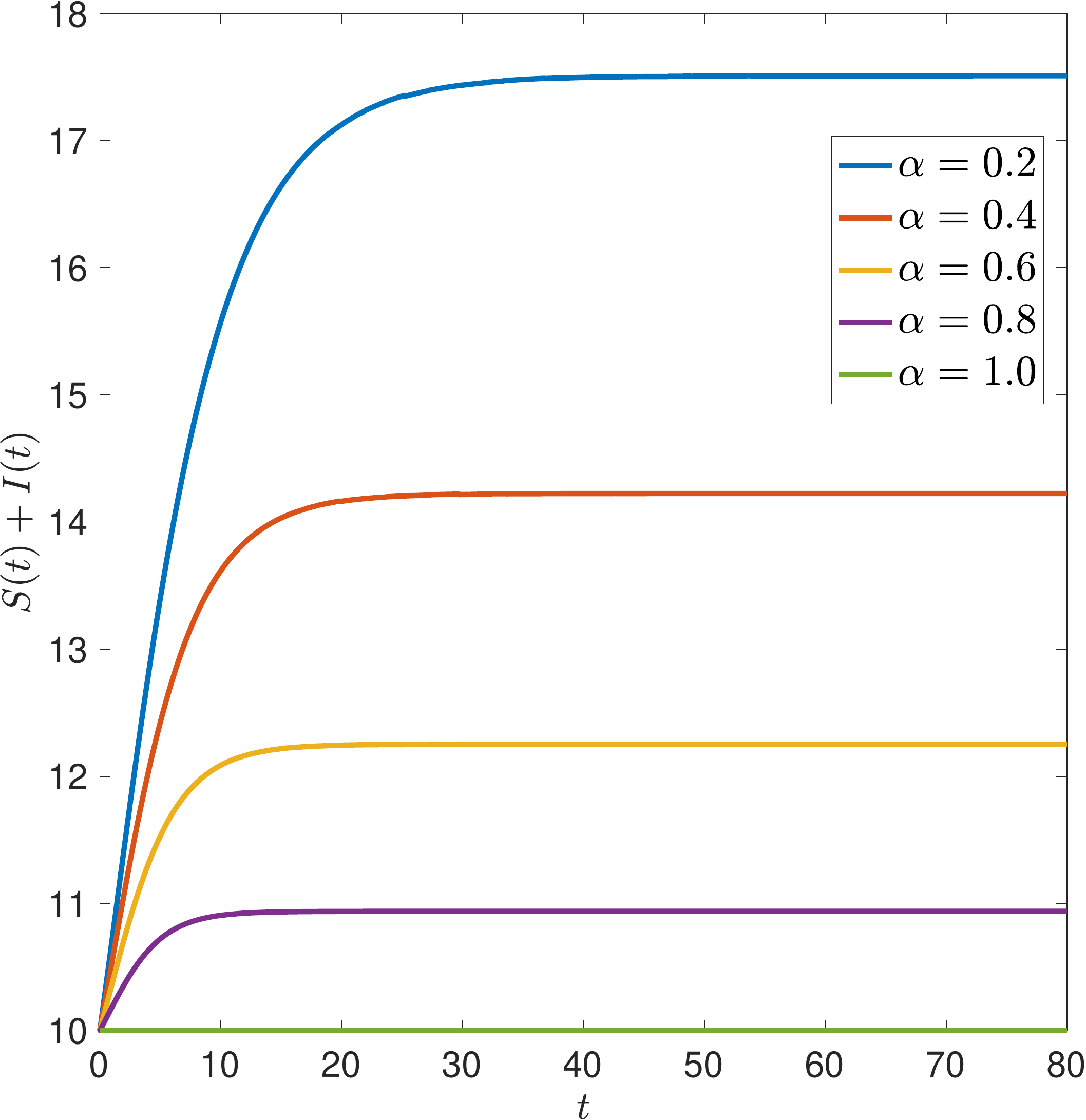}
}
\caption{Numerical solutions to \eqref{SIS} for  $\beta=0.7$, $\gamma=0.2$, $S_{0}=6$, $I_{0}=4$, and $M(\alpha)\equiv1$ as $\alpha$ changes.}
\label{fig:CF1}
\end{figure}

\begin{figure}[h!]
\subfloat[][$S(t)$]{
\includegraphics[width=0.3\columnwidth]{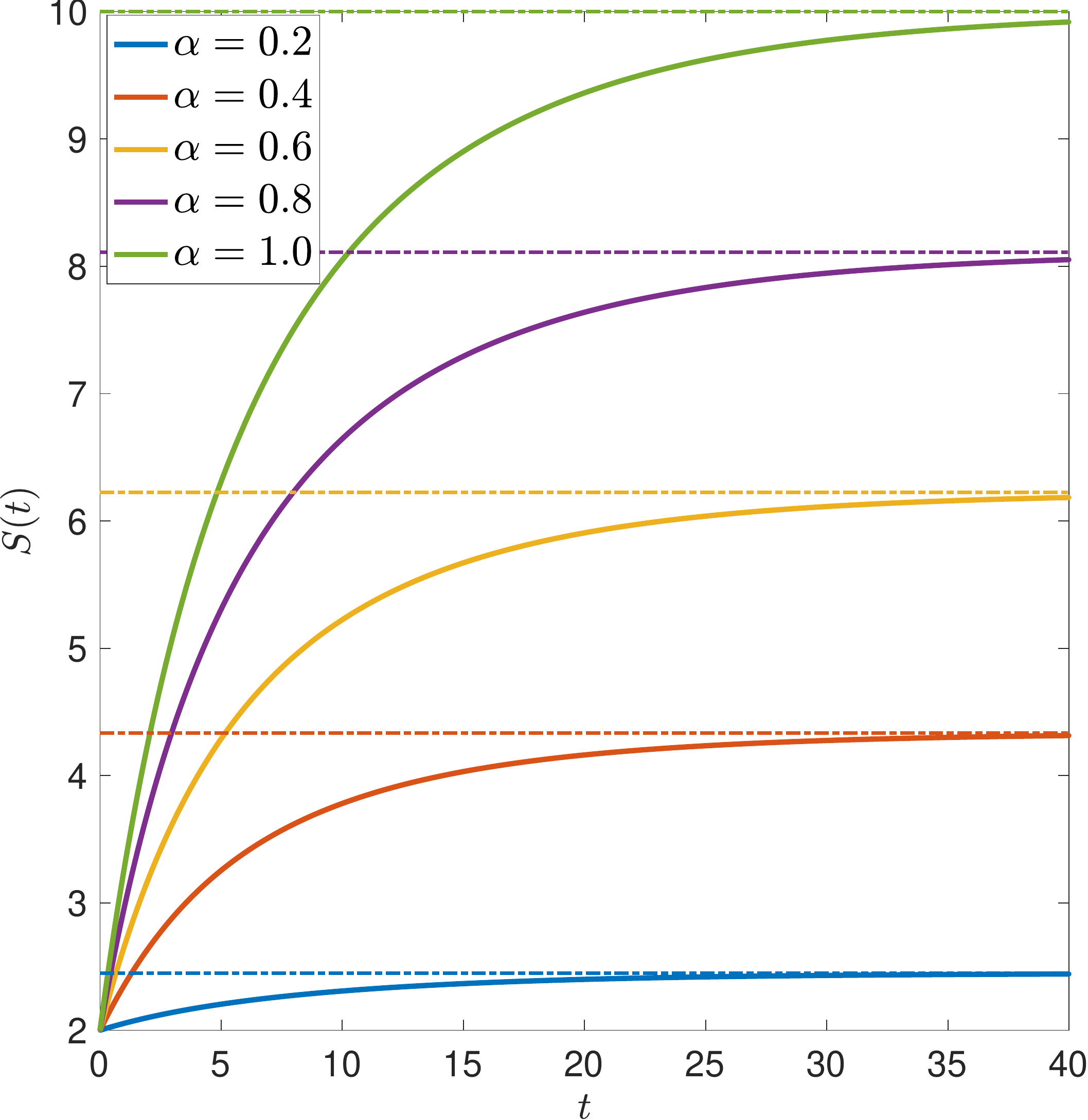}
}
\quad
\subfloat[][$I(t)$]{
\includegraphics[width=0.3\columnwidth]{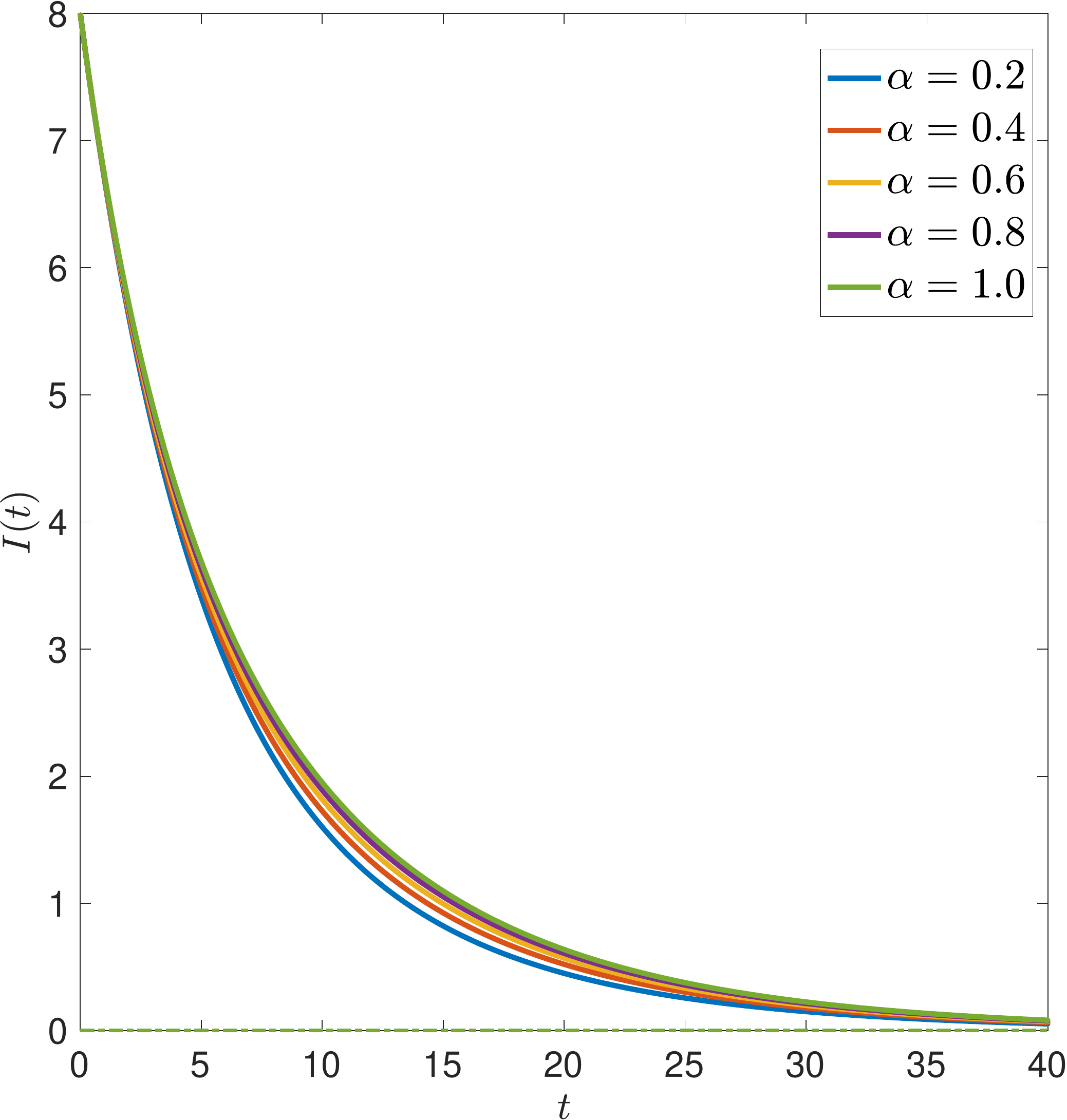}
}
\quad
\subfloat[][$S(t)+I(t)$]{
\includegraphics[width=0.3\columnwidth]{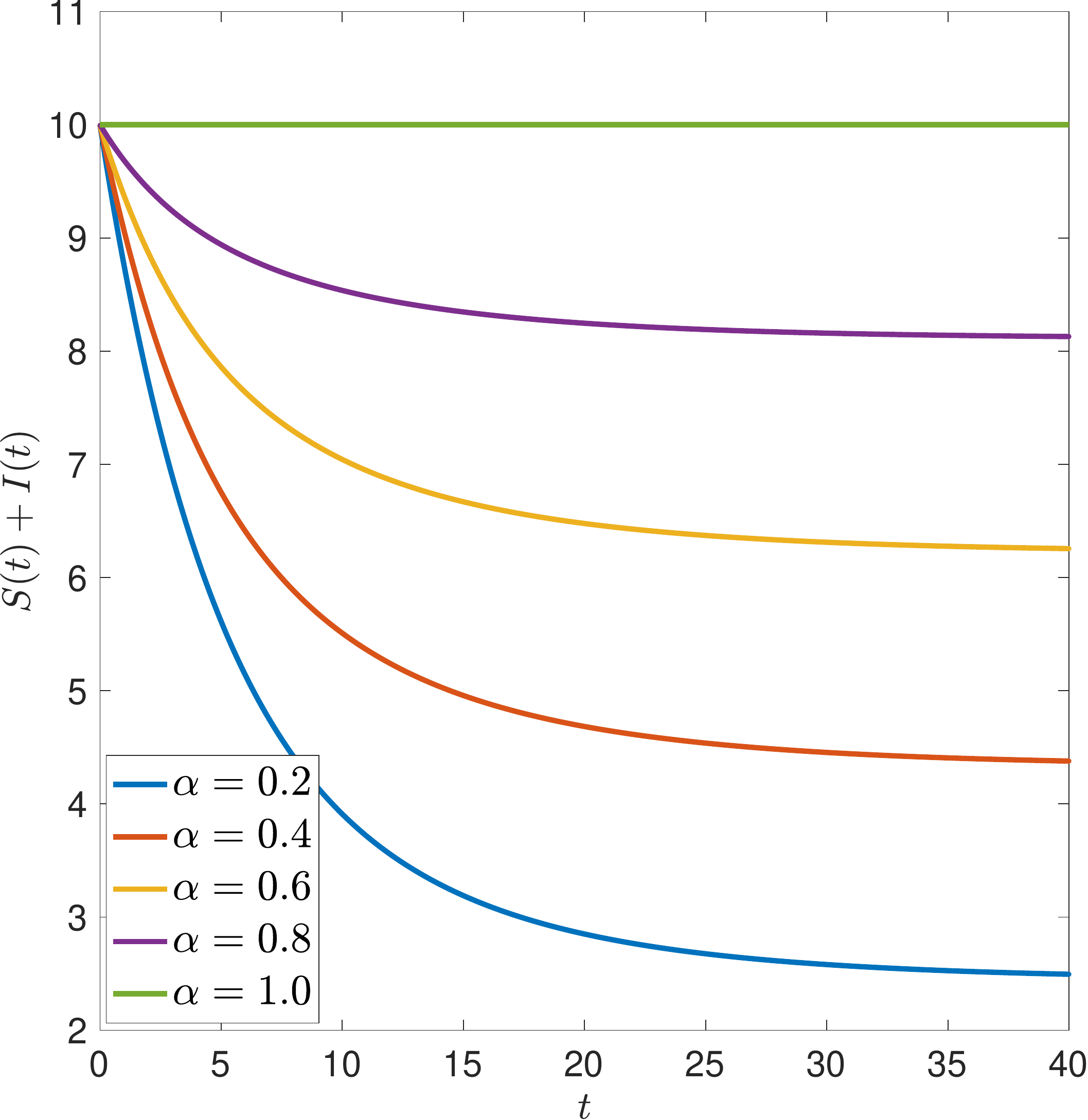}
}
\caption{Numerical solutions to \eqref{SIS} with   $\beta=0.1$, $\gamma=0.2$, $S_{0}=6$, $I_{0}=4$, and $M(\alpha)\equiv1$ as $\alpha$ changes.}
\label{fig:CF2}
\end{figure}

In Figures \ref{fig:CF1} and \ref{fig:CF2} we show the results obtained as $\alpha$ changes in $\{0.2,0.4,0.6,0.8,1\}$. The parameters $S_{0}$, $I_{0}$, $M(\alpha)$ and $\gamma$ are the same in all simulations, fixed as $S_{0}=6$, $I_{0}=4$, $M(\alpha)\equiv1$ and $\gamma=0.2$. The parameter $\beta$, instead, is set to $\beta=0.7$ in Figure \ref{fig:CF1}, which implies $R>1$, while $\beta=0.1$ in Figure \ref{fig:CF2}, and thus $R<1$. The dashed lines represent the equilibria estimated in Proposition \ref{pequilibria}. Both $S$ and $I$ monotonically converge to their equilibria. Note that the time delay induced by the fractional order emerges in the fact that the smallest is $\alpha$, the slowest is the convergence of the related solutions to equilibria. Finally remark  that for $\alpha=1$ the sum $N(t)=S(t)+I(t)$ is constant, since \eqref{SIS} coincides with the classical SIS model, where if $\alpha<1$ then the monotonicity of $N(t)$ is in agreement with Theorem \ref{pN}.

\begin{figure}[h!]
\subfloat[][$S(t)$]{
\includegraphics[width=0.3\columnwidth]{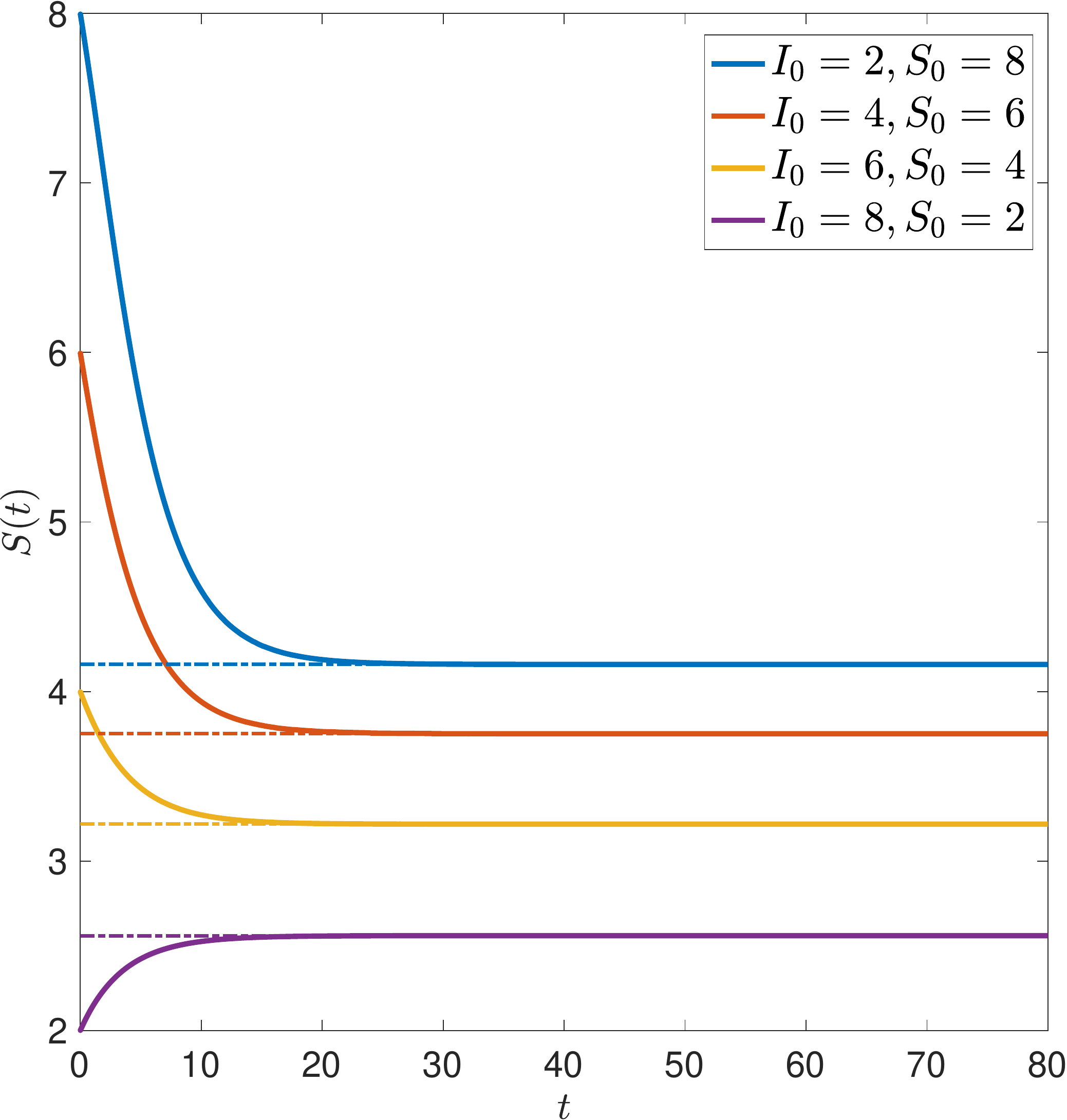}
}
\quad
\subfloat[][$I(t)$]{
\includegraphics[width=0.3\columnwidth]{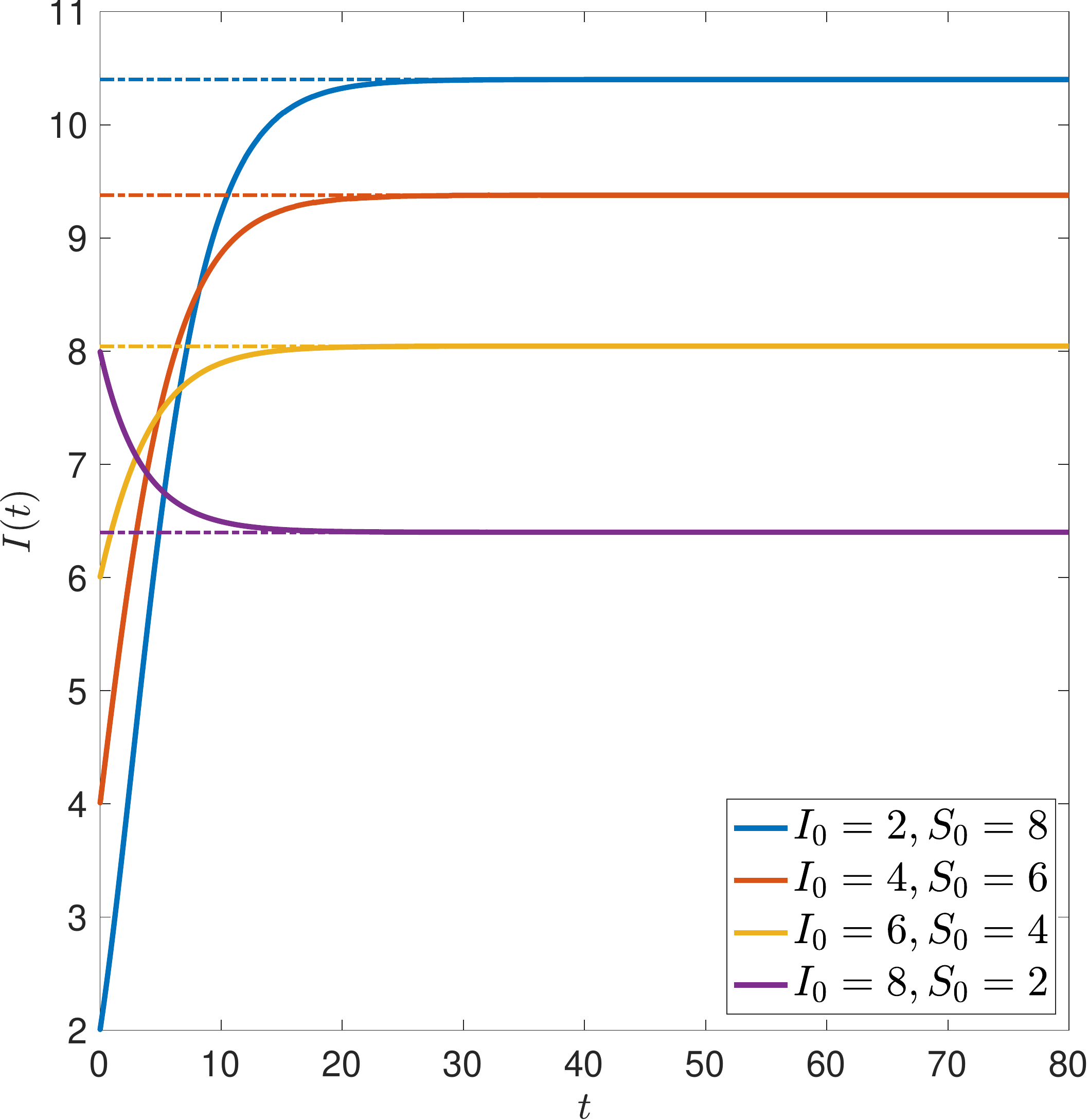}
}
\quad
\subfloat[][$S(t)+I(t)$]{
\includegraphics[width=0.3\columnwidth]{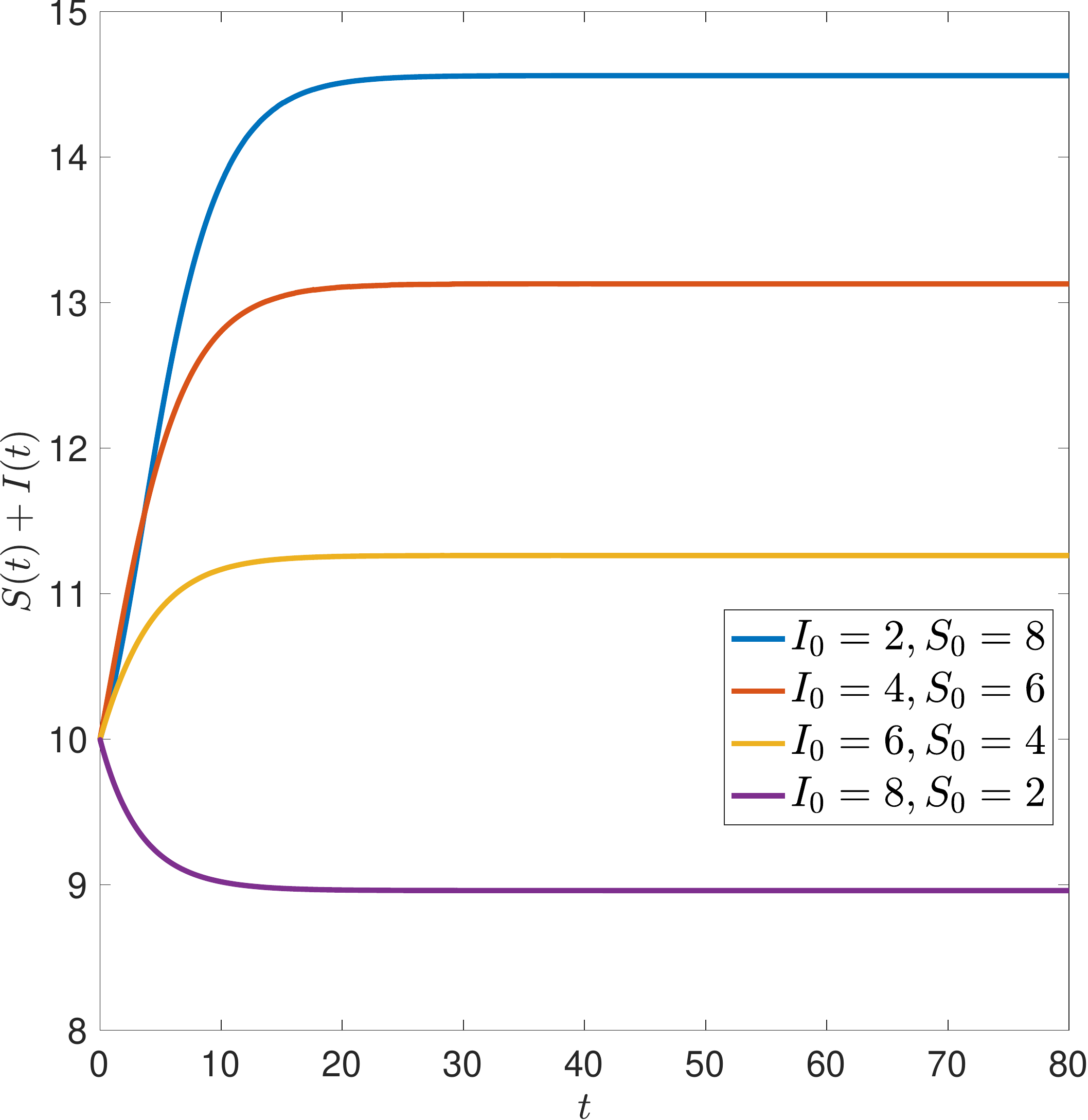}
}
\caption{Numerical solutions  to \eqref{SIS} for $\alpha=0.5$, $\beta=0.7$, $\gamma=0.2$ and $M(\alpha)\equiv1$ as $I_{0}$ and $S_{0}$ change.}
\label{fig:CF3}
\end{figure}
\begin{figure}[h!]
\subfloat[][$S(t)$]{
\includegraphics[width=0.3\columnwidth]{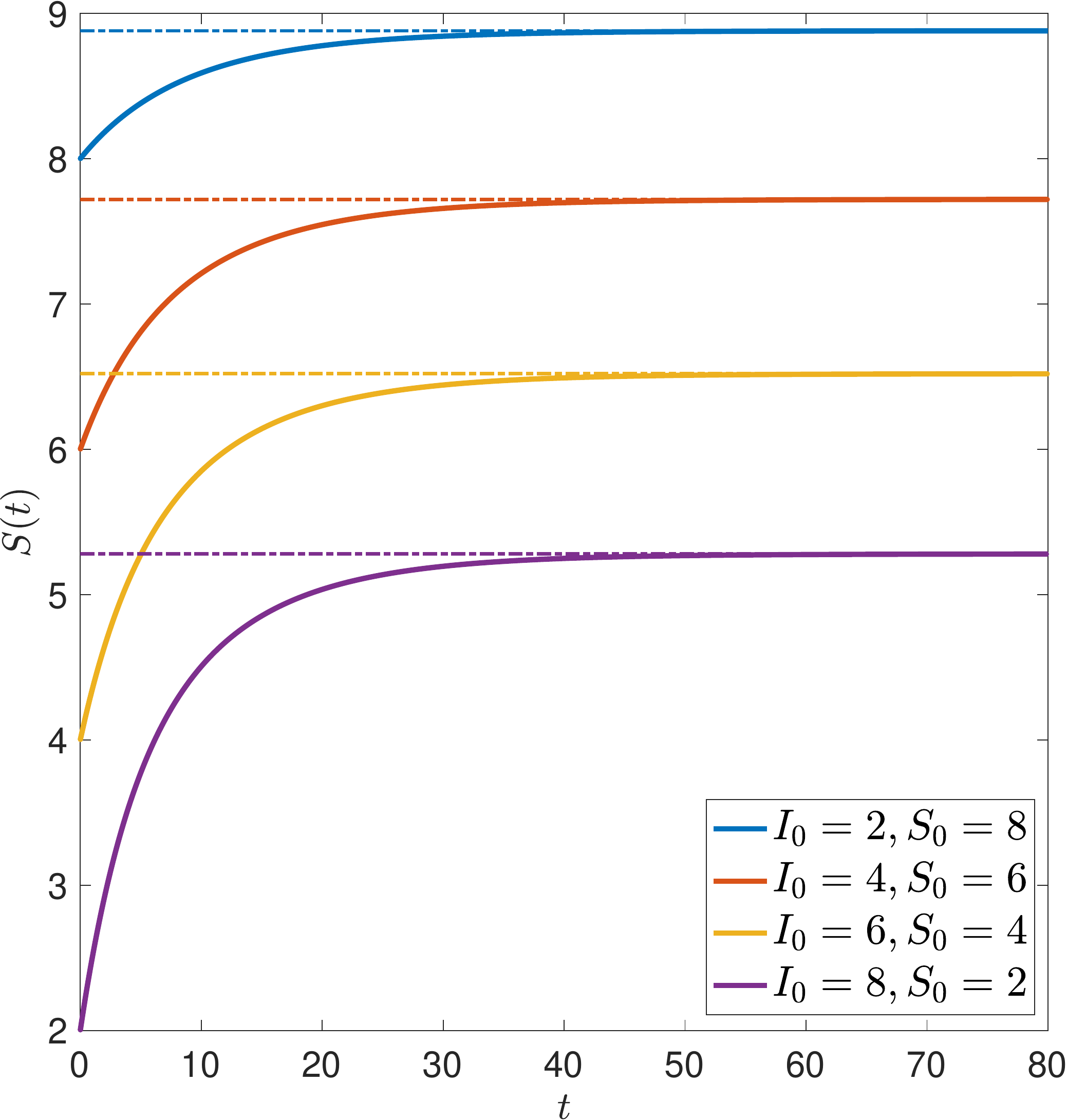}
}
\quad
\subfloat[][$I(t)$]{
\includegraphics[width=0.3\columnwidth]{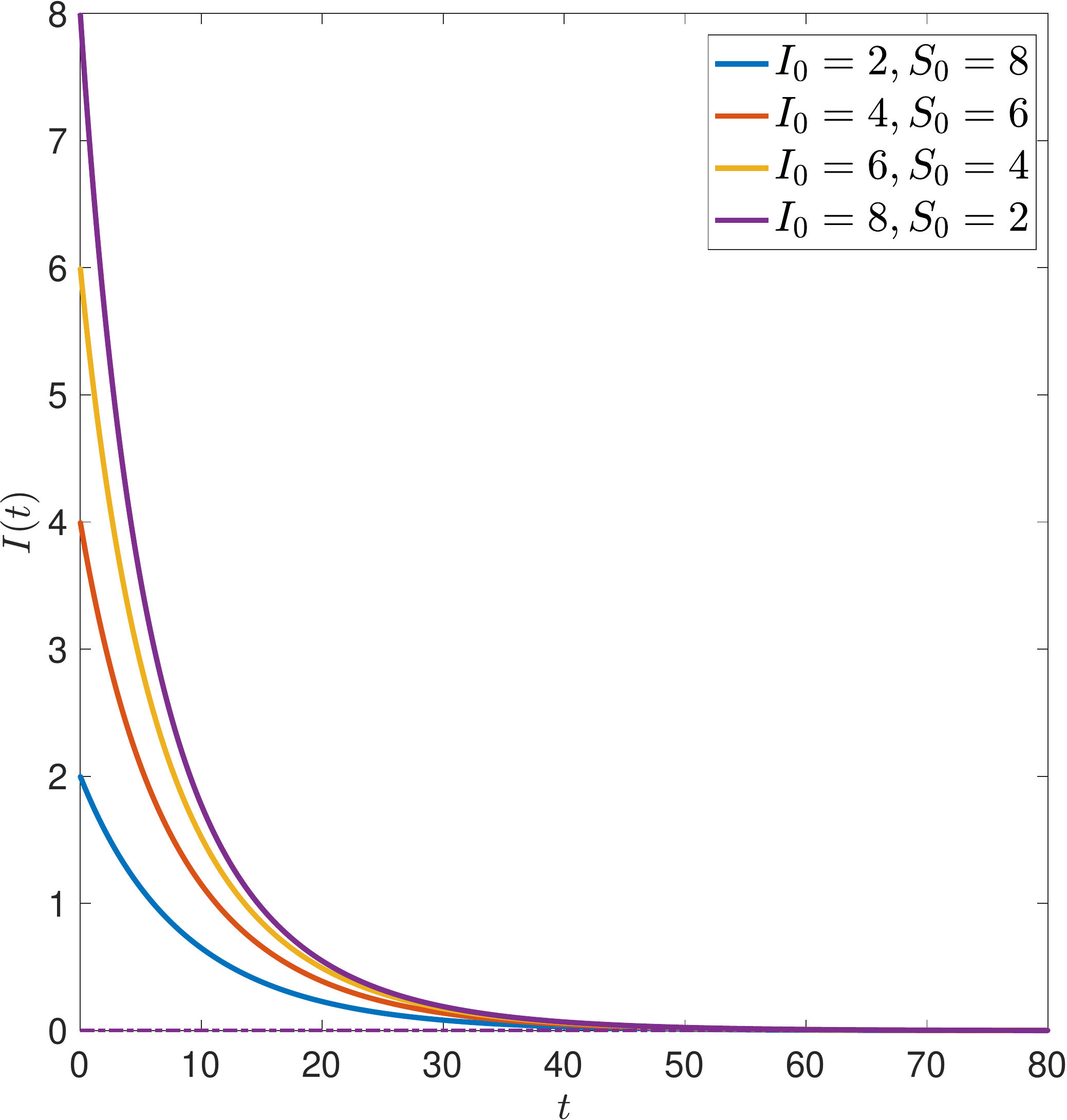}
}
\quad
\subfloat[][$S(t)+I(t)$]{
\includegraphics[width=0.3\columnwidth]{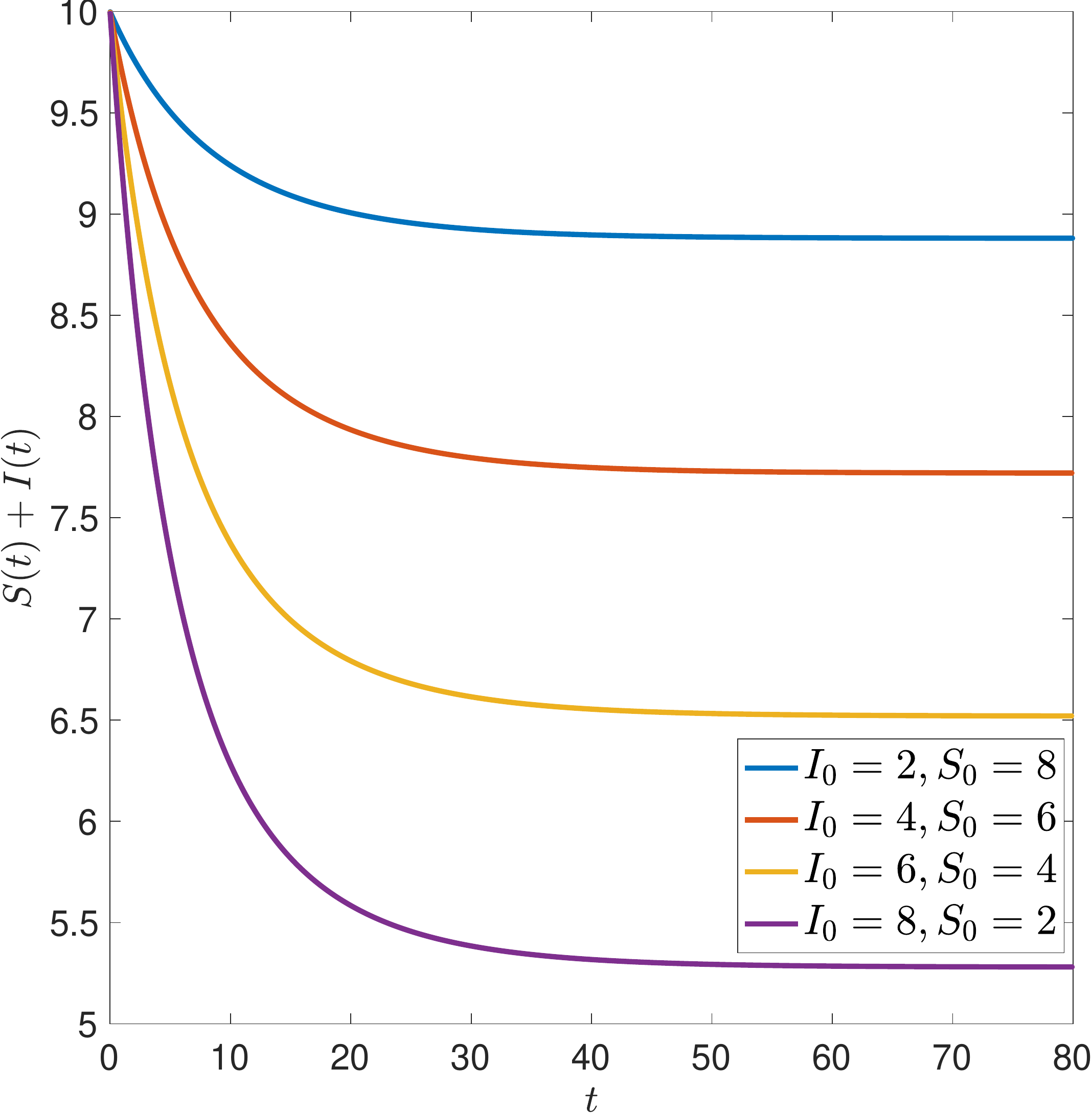}
}
\caption{Numerical solutions to \eqref{SIS} for $\alpha=0.5$, $\beta=0.1$, $\gamma=0.2$ and $M(\alpha)\equiv1$ as $I_{0}$ and $S_{0}$ change.}
\label{fig:CF4}
\end{figure}

In Figures \ref{fig:CF3} and \ref{fig:CF4} we fix $\alpha=0.5$ and we set $S_0=10-I_0$ and let vary the initial data $I_{0}$ in $\{2,4,6,8\}$. We note that the qualitative behavior of the solutions is in agreement with the theory developed here and it is not much affected by initial data.

}

\subsubsection{Comparison between Caputo SIS model and Caputo-Fabrizio SIS model}

\begin{figure}[h!]
\subfloat[][$S(t),\,\alpha=\alpha_{1}=0.2$]{
\includegraphics[width=0.3\columnwidth]{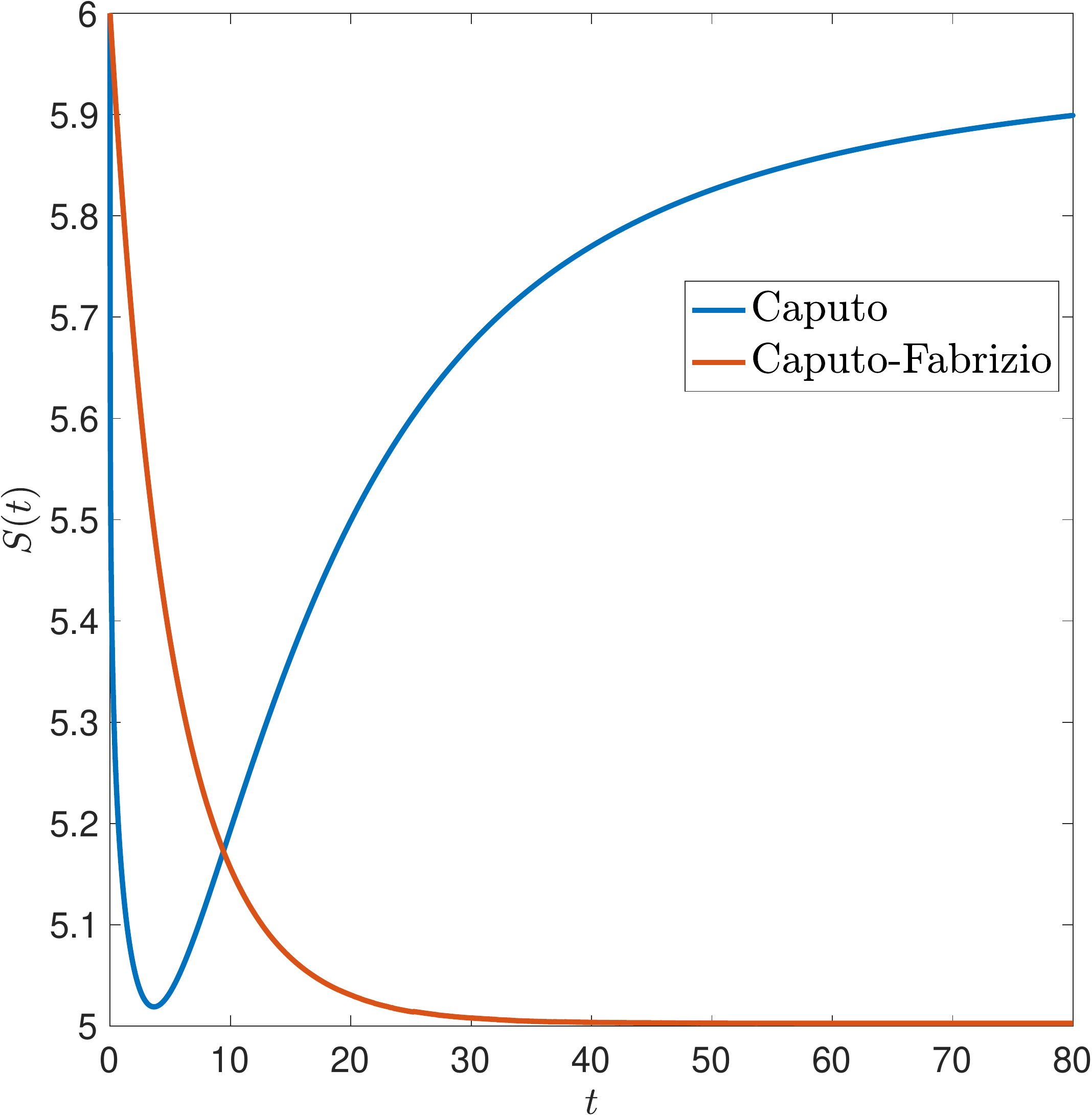}
}
\quad
\subfloat[][$I(t),\, \alpha_{2}=1$]{
\includegraphics[width=0.3\columnwidth]{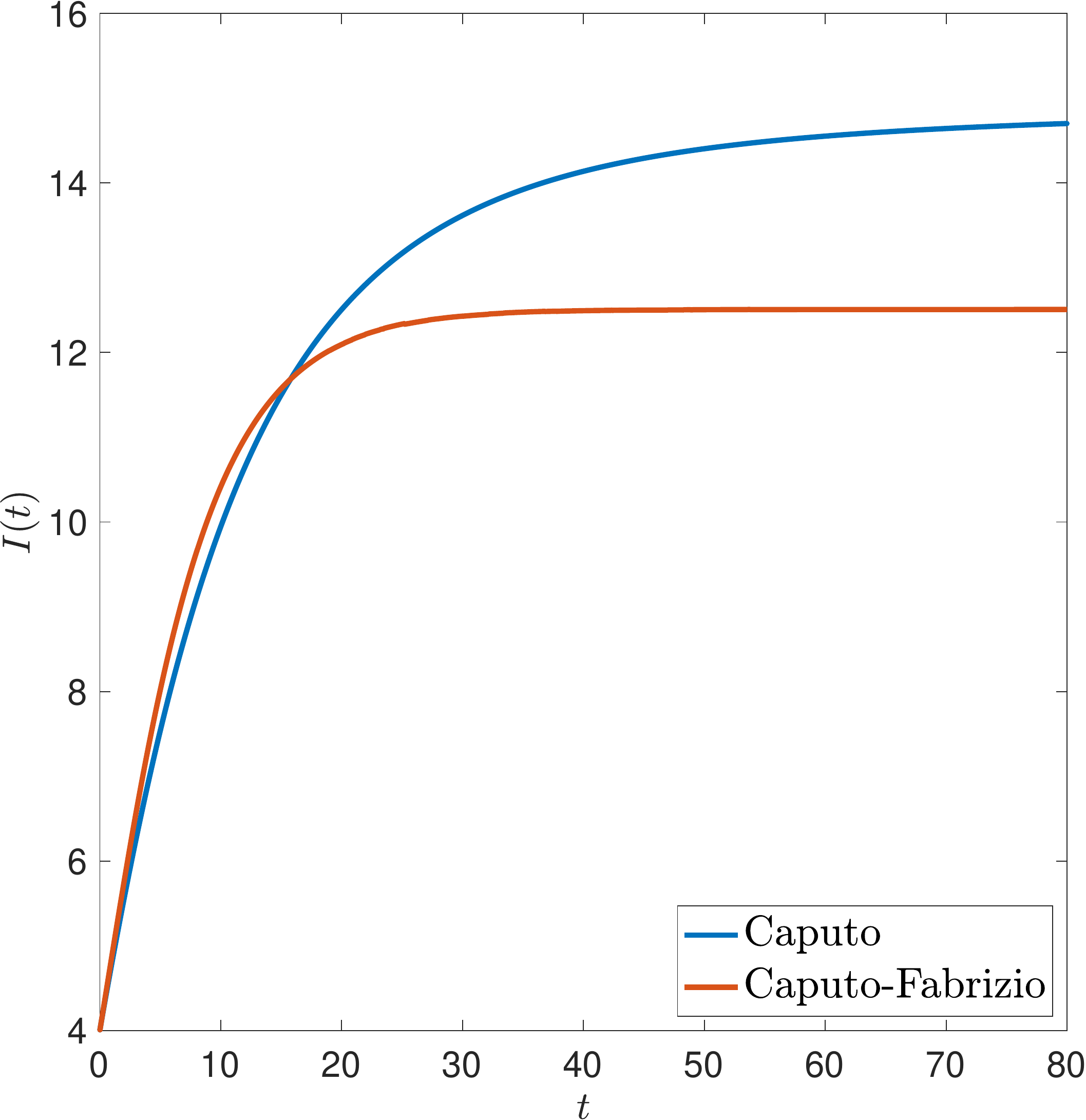}
}
\quad
\subfloat[][$S(t)+I(t)$]{
\includegraphics[width=0.3\columnwidth]{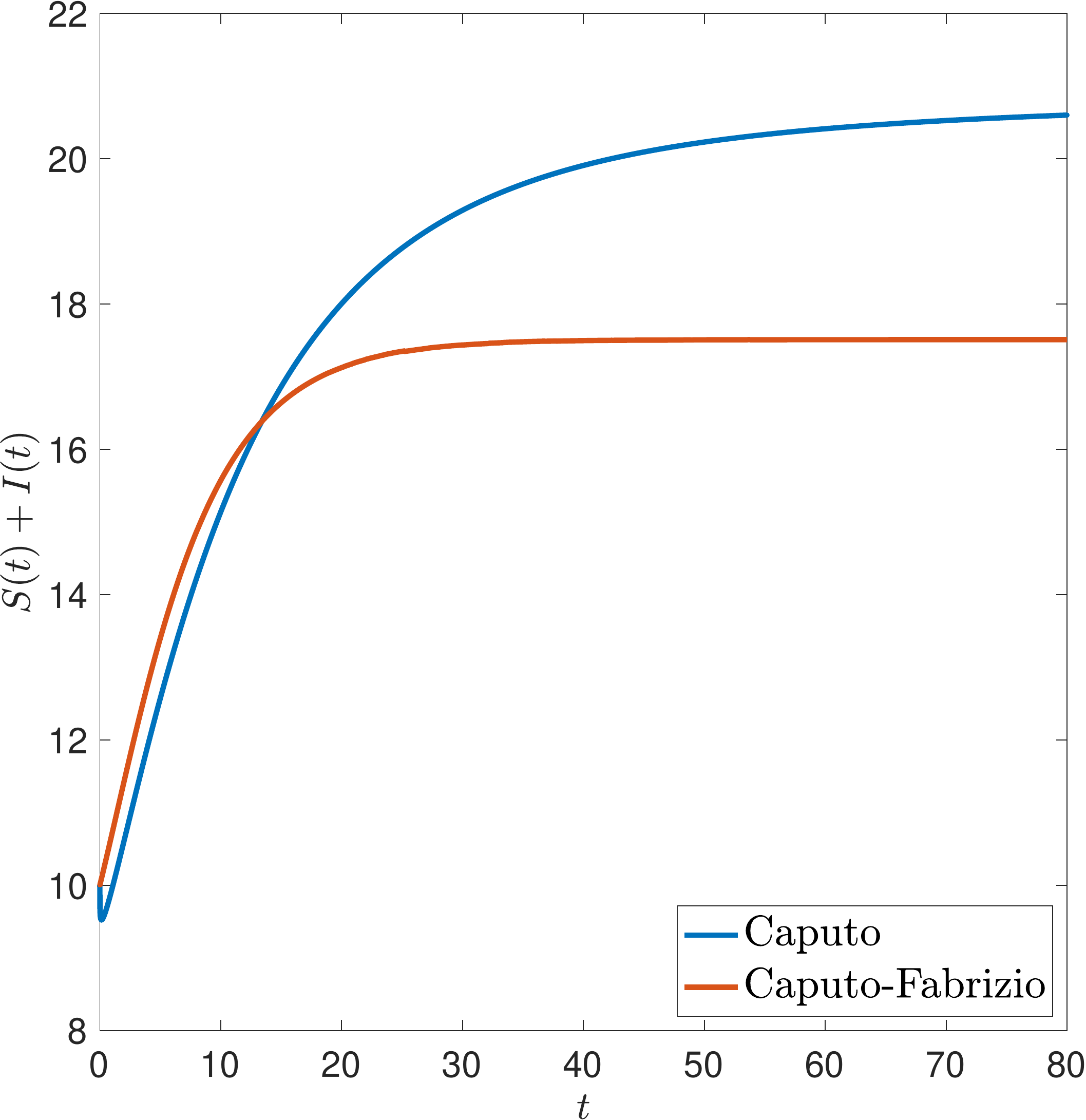}
}\\
\subfloat[][$S(t),\, \alpha=\alpha_{1}=0.5$]{
\includegraphics[width=0.3\columnwidth]{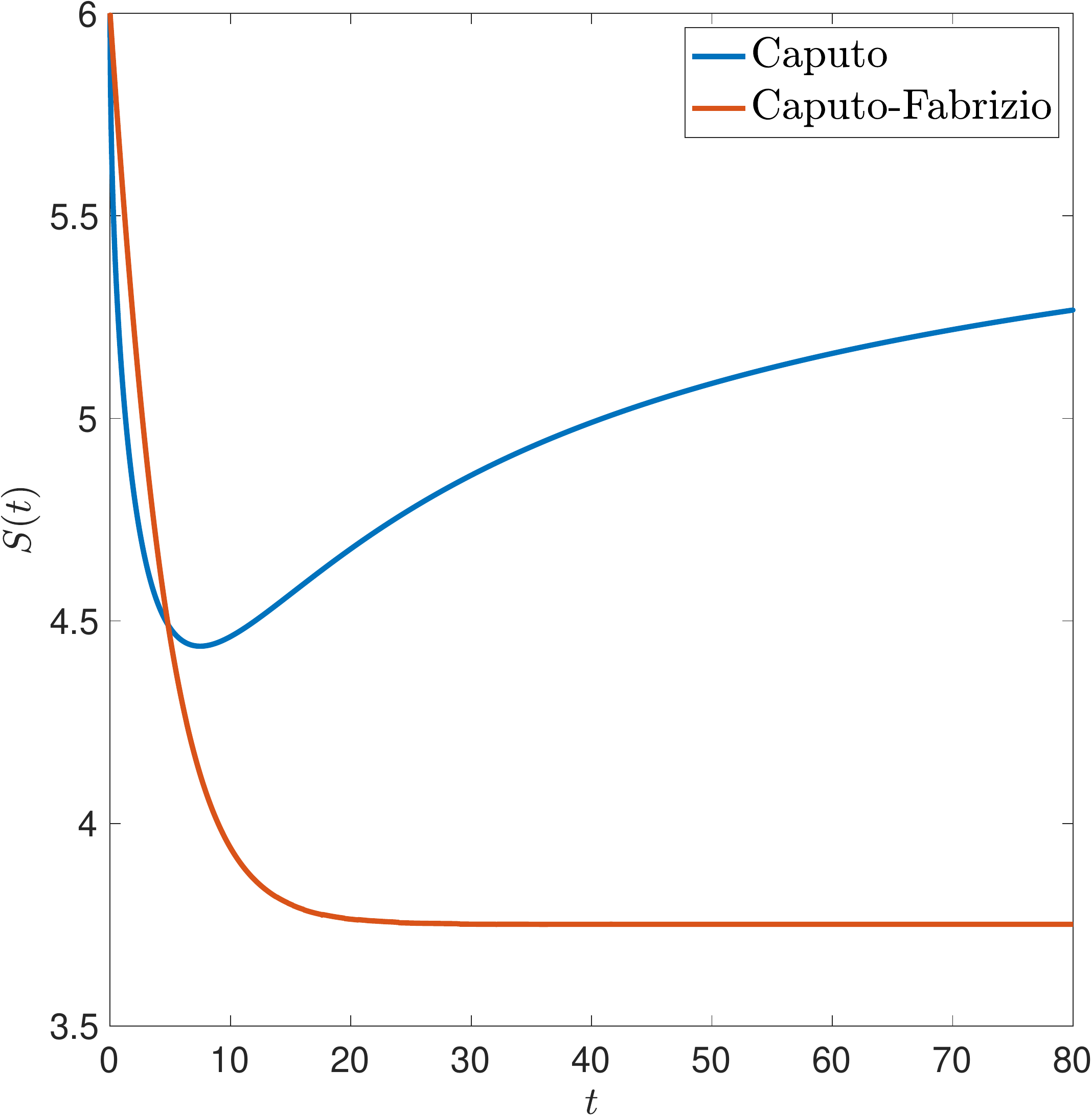}
}
\quad
\subfloat[][$I(t),\, \alpha_{2}=1$]{
\includegraphics[width=0.3\columnwidth]{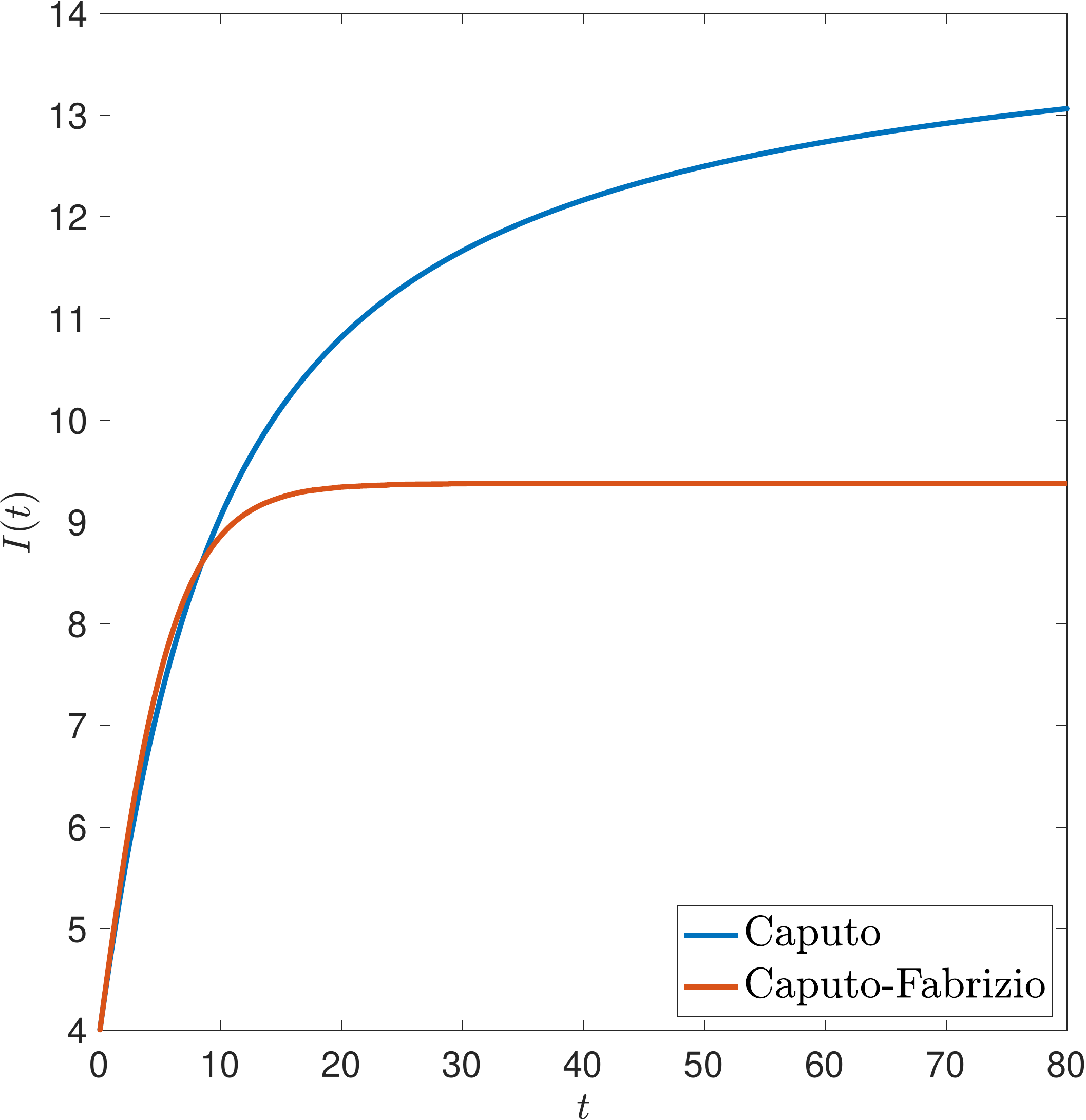}
}
\quad
\subfloat[][$S(t)+I(t)$]{
\includegraphics[width=0.3\columnwidth]{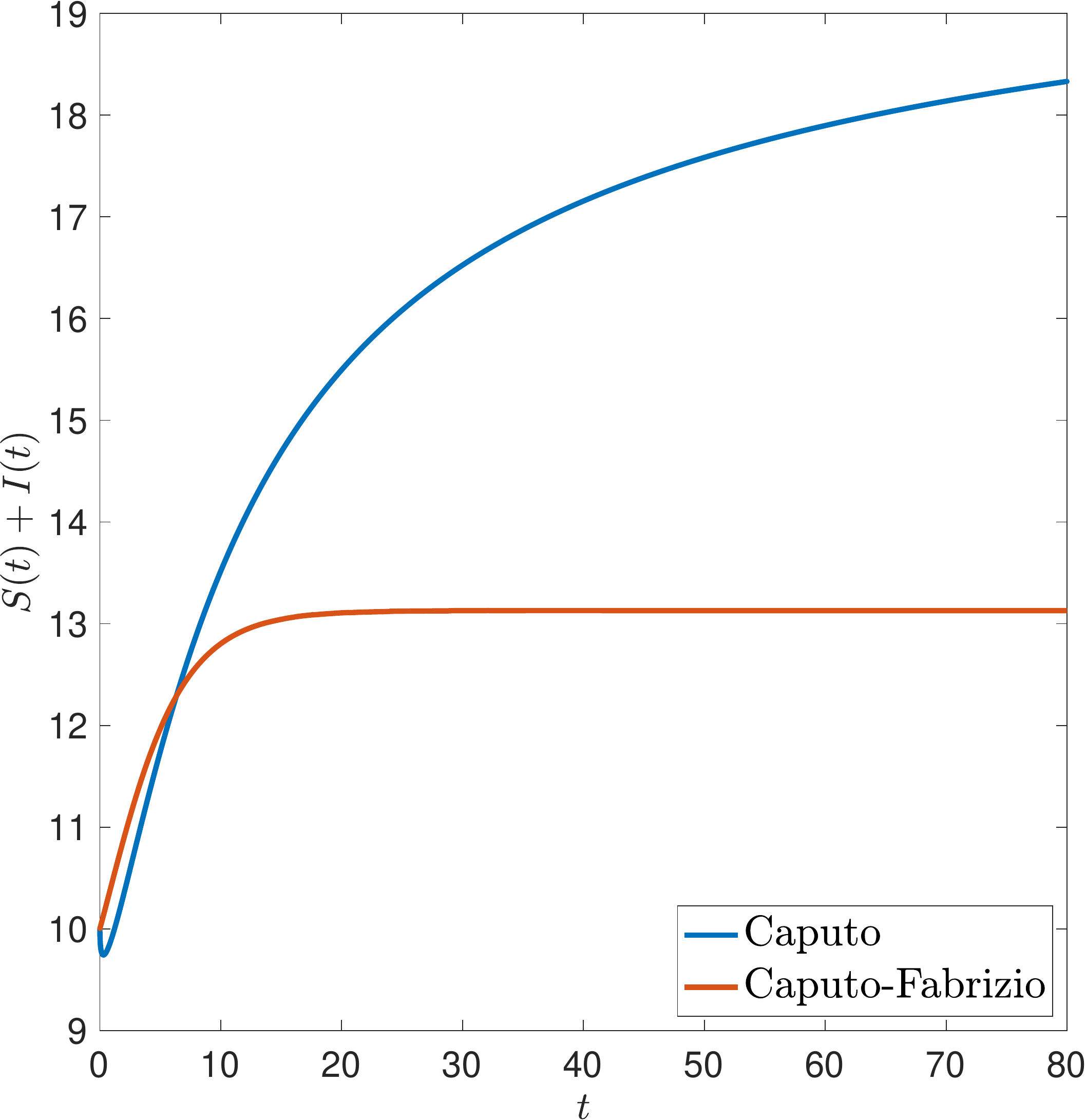}
}\\
\subfloat[][$S(t), \alpha=\alpha_{1}=0.8$]{
\includegraphics[width=0.3\columnwidth]{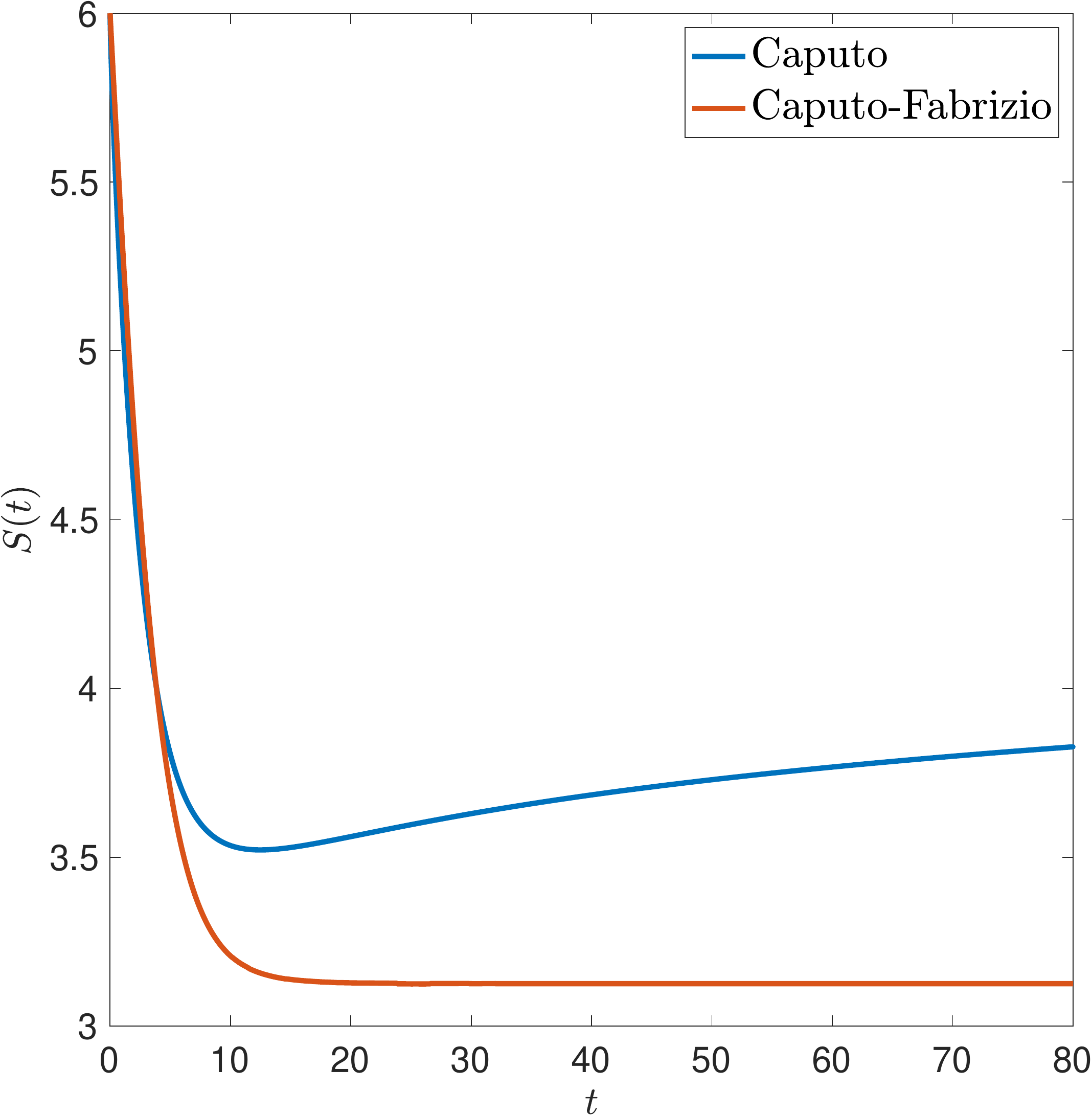}
}
\quad
\subfloat[][$I(t),\, \alpha_{2}=1$]{
\includegraphics[width=0.3\columnwidth]{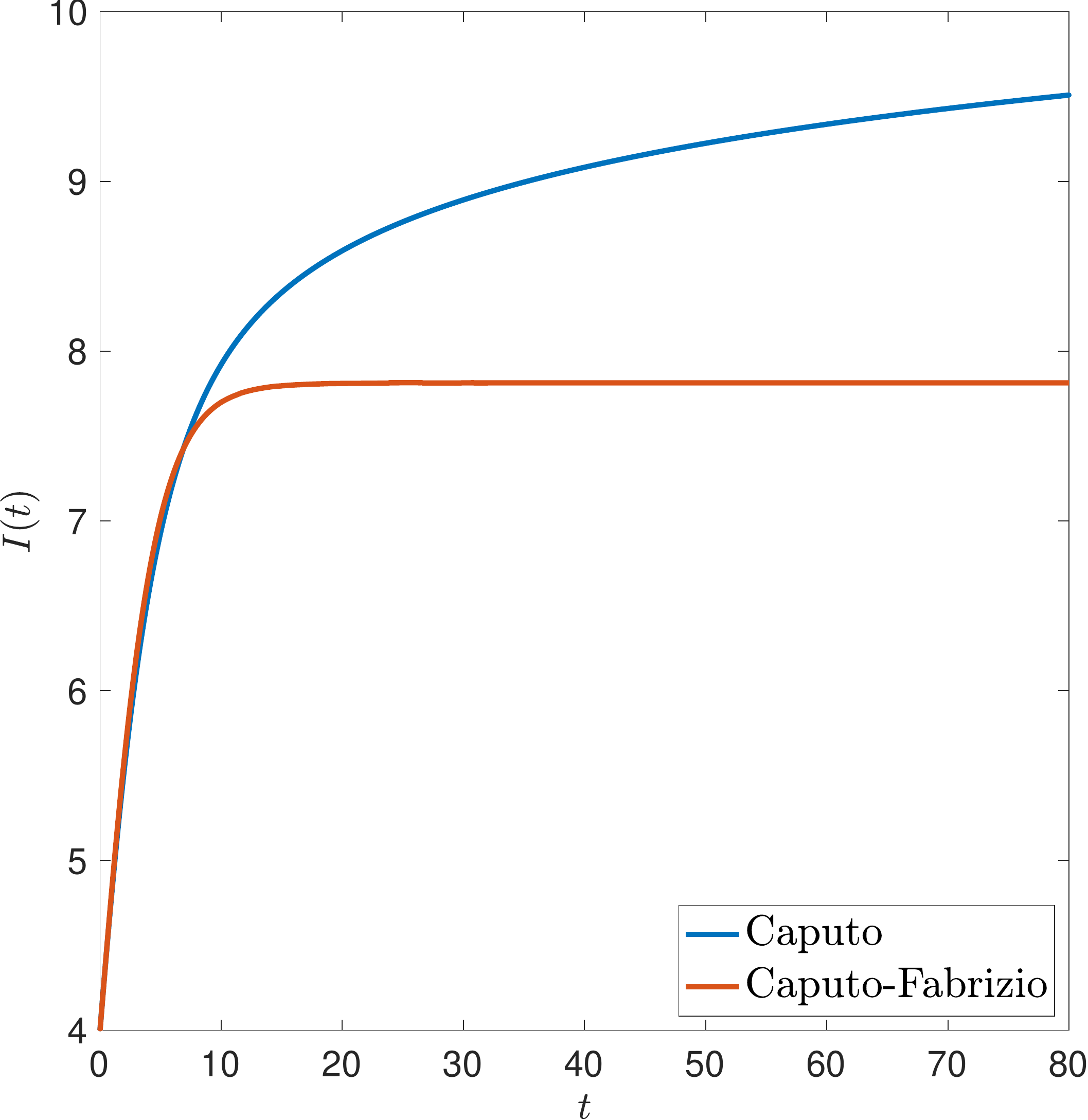}
}
\quad
\subfloat[][$S(t)+I(t)$]{
\includegraphics[width=0.3\columnwidth]{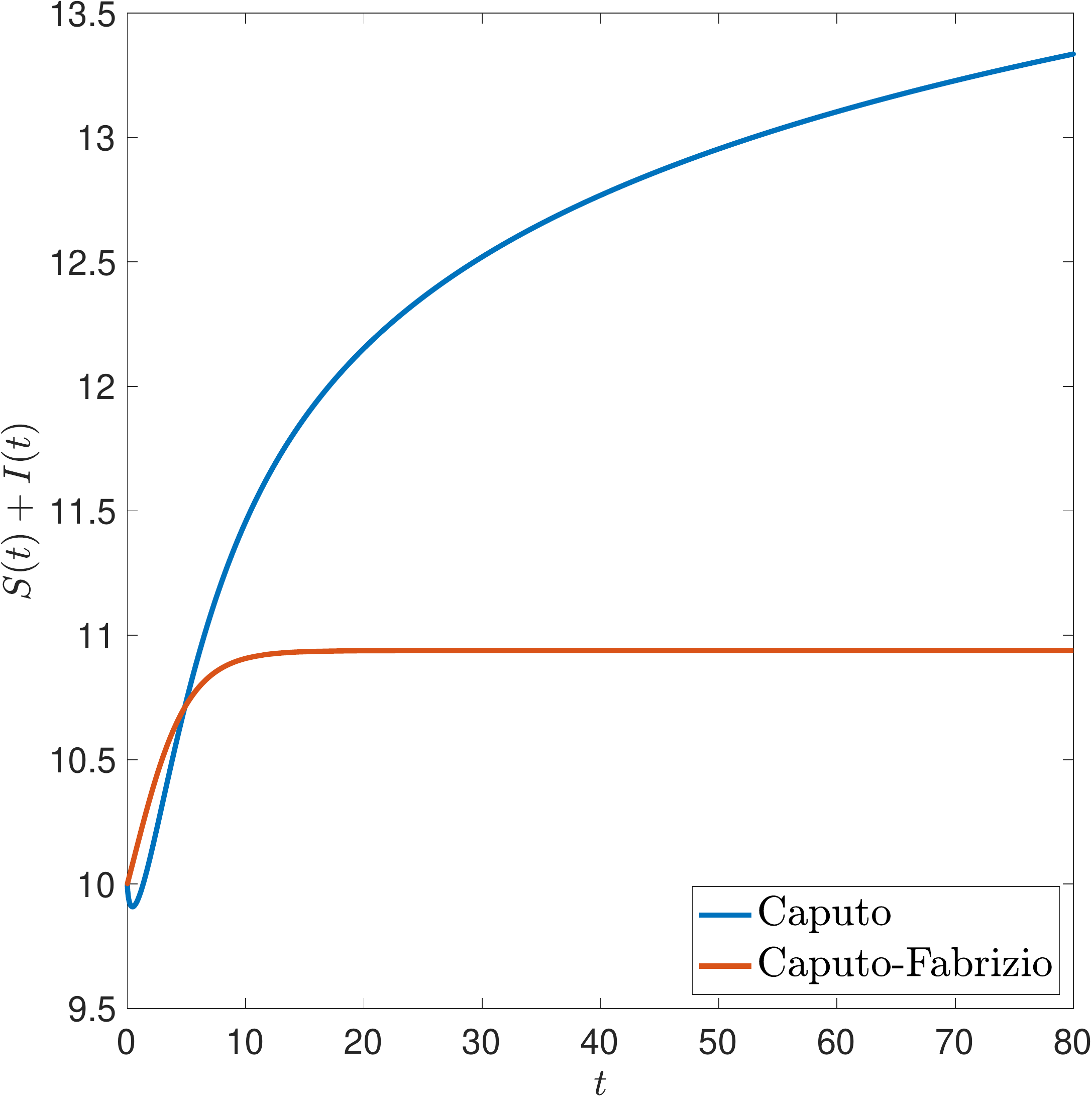}
}
\caption{Numerical solutions to \eqref{modello3} and \eqref{SIS} for $\alpha=\alpha_{1}=0.2$ (top), $\alpha=\alpha_{1}=0.5$ (center) and $\alpha=\alpha_{1}=0.8$ (bottom) with $\alpha_{2}=1$, $\beta=0.7$, $\gamma=0.2$,  $S_{0}=6$, $I_{0}=4$ and $M(\alpha)\equiv1$.}
\label{fig:CFCC1}
\end{figure}

\begin{figure}[h!]
\subfloat[][$S(t), \,\alpha=\alpha_{1}=0.2$]{
\includegraphics[width=0.3\columnwidth]{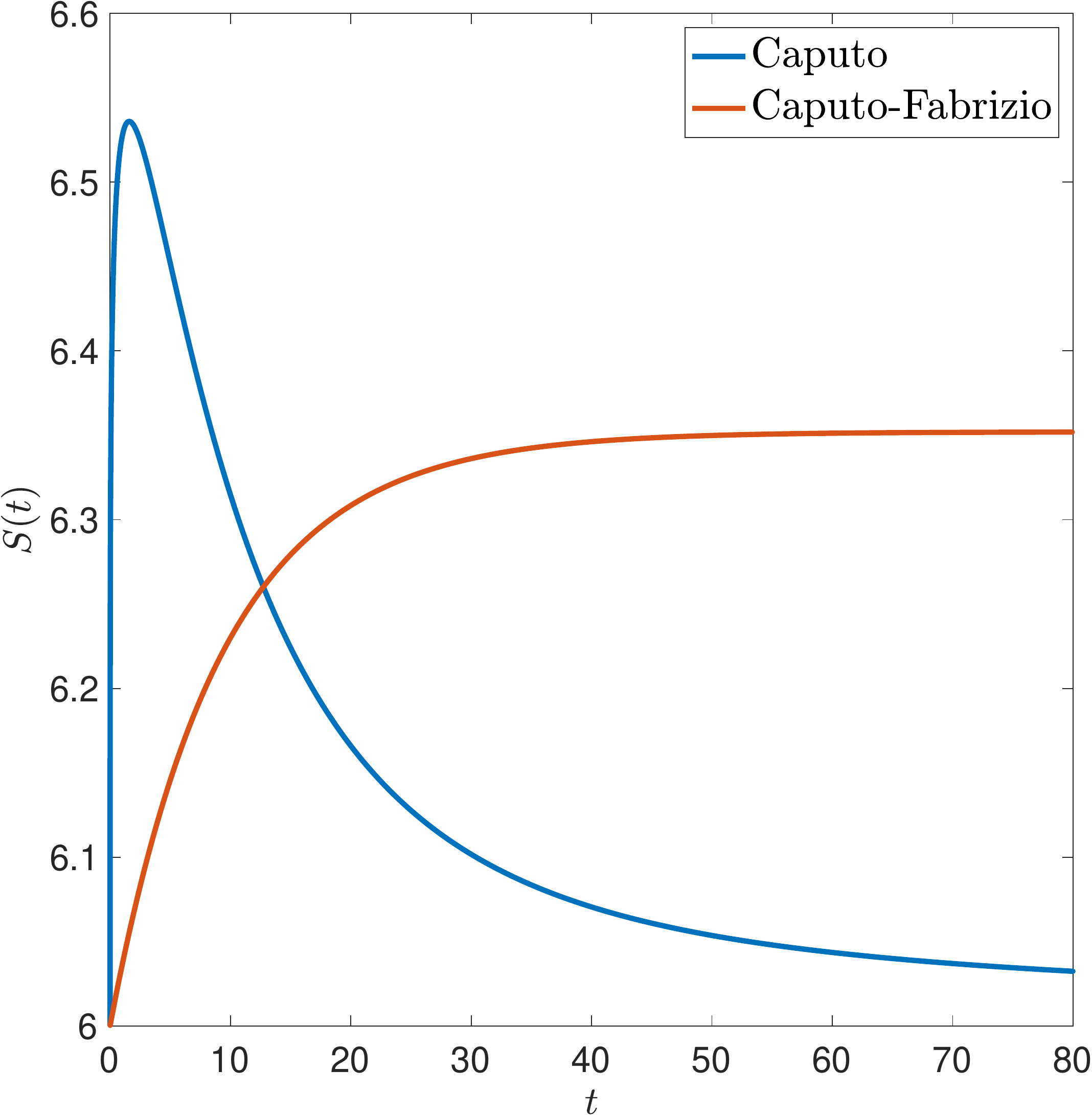}
}
\quad
\subfloat[][$I(t),\, \alpha_{2}=1$]{
\includegraphics[width=0.3\columnwidth]{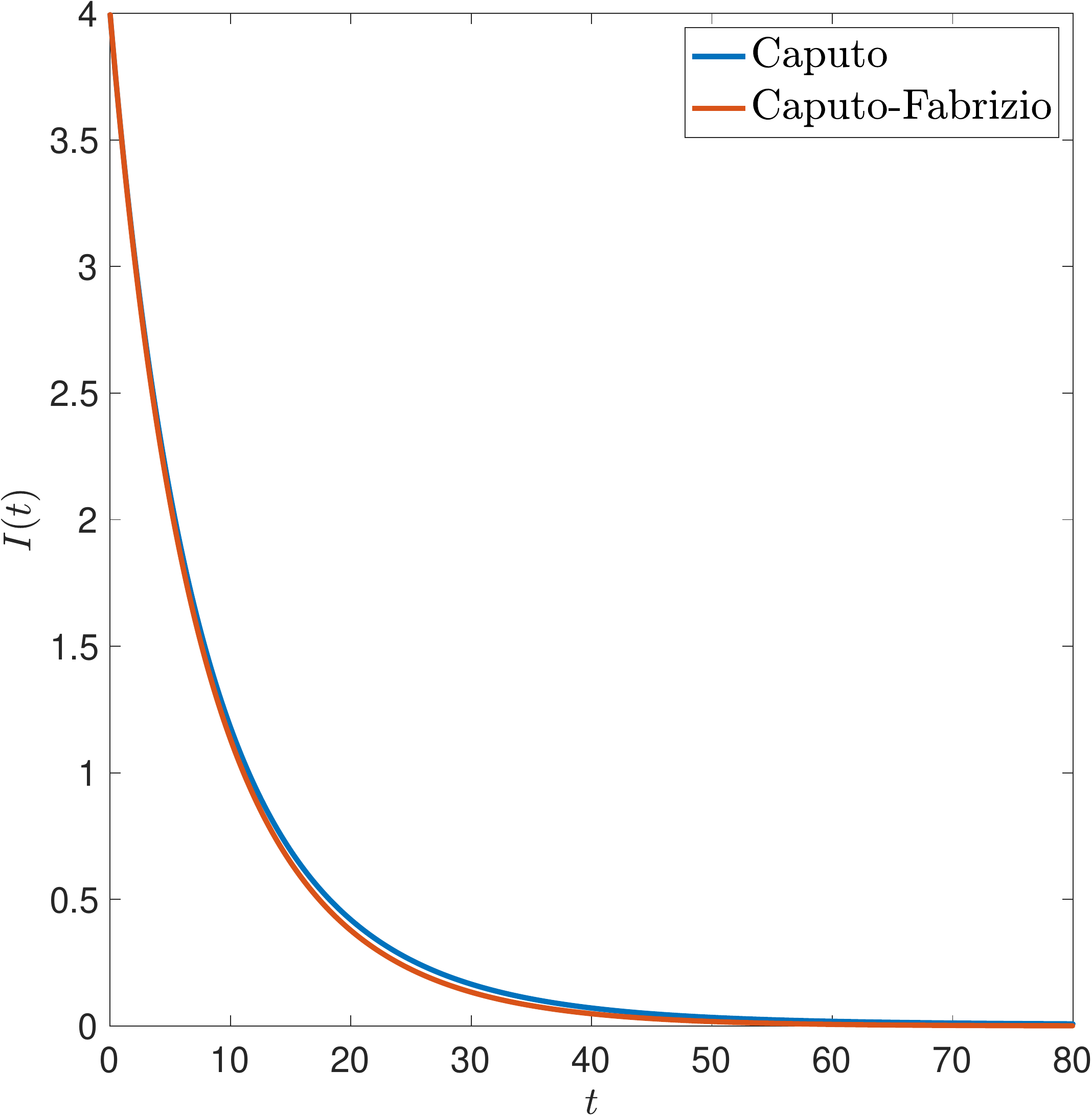}
}
\quad
\subfloat[][$S(t)+I(t)$]{
\includegraphics[width=0.3\columnwidth]{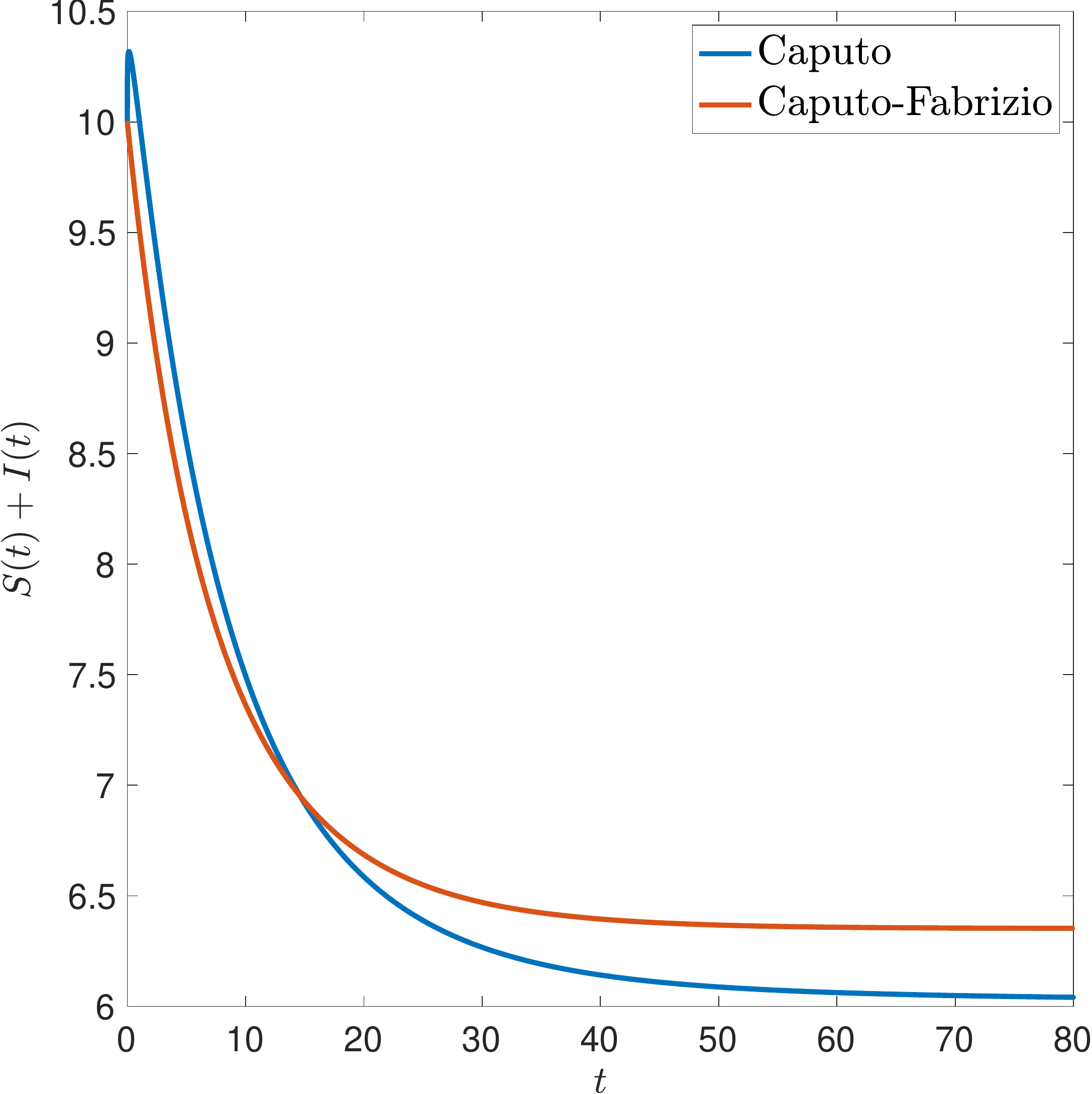}
}\\
\subfloat[][$S(t), \,\alpha=\alpha_{1}=0.5$]{
\includegraphics[width=0.3\columnwidth]{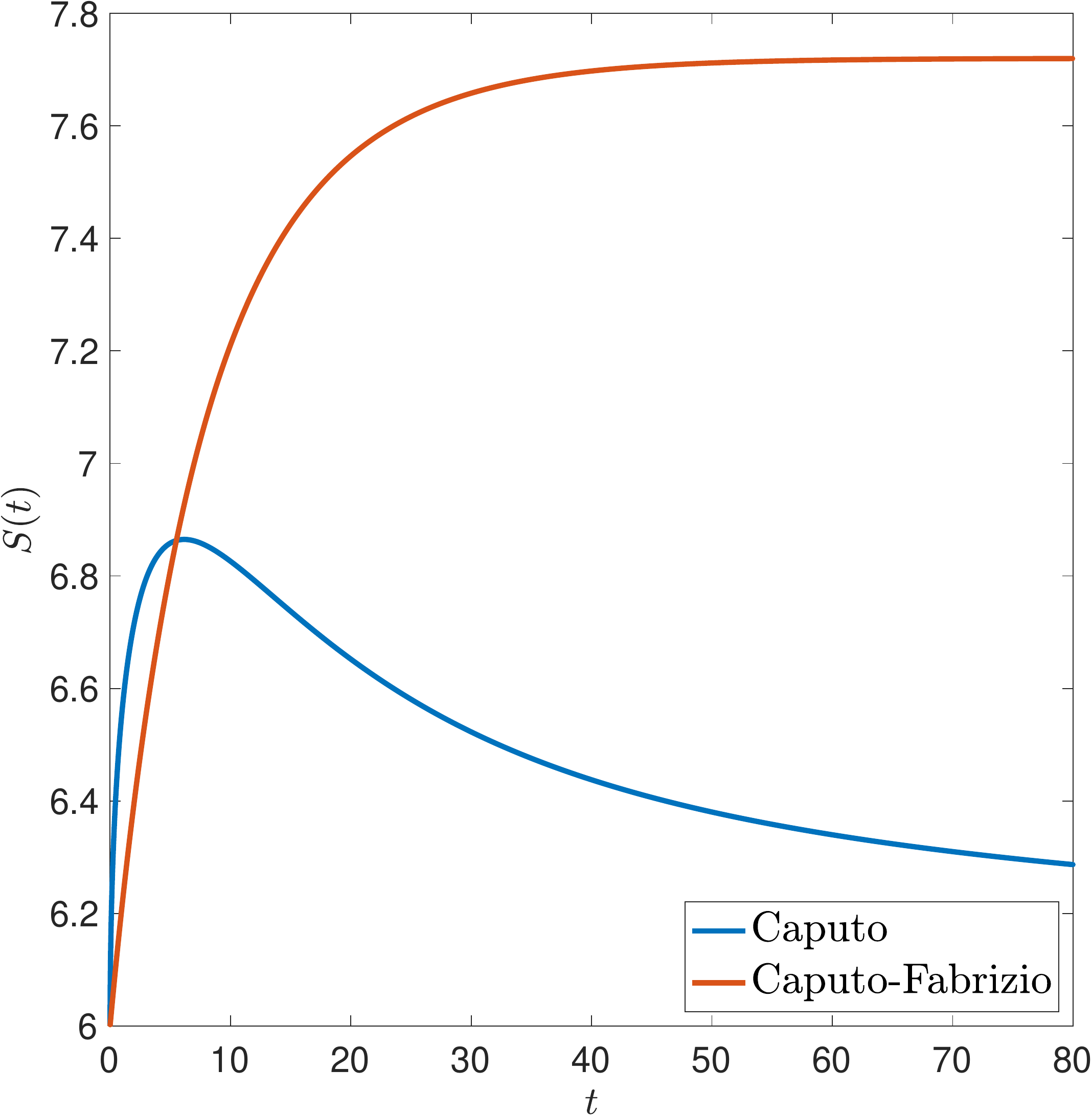}
}
\quad
\subfloat[][$I(t),\, \alpha_{2}=1$]{
\includegraphics[width=0.3\columnwidth]{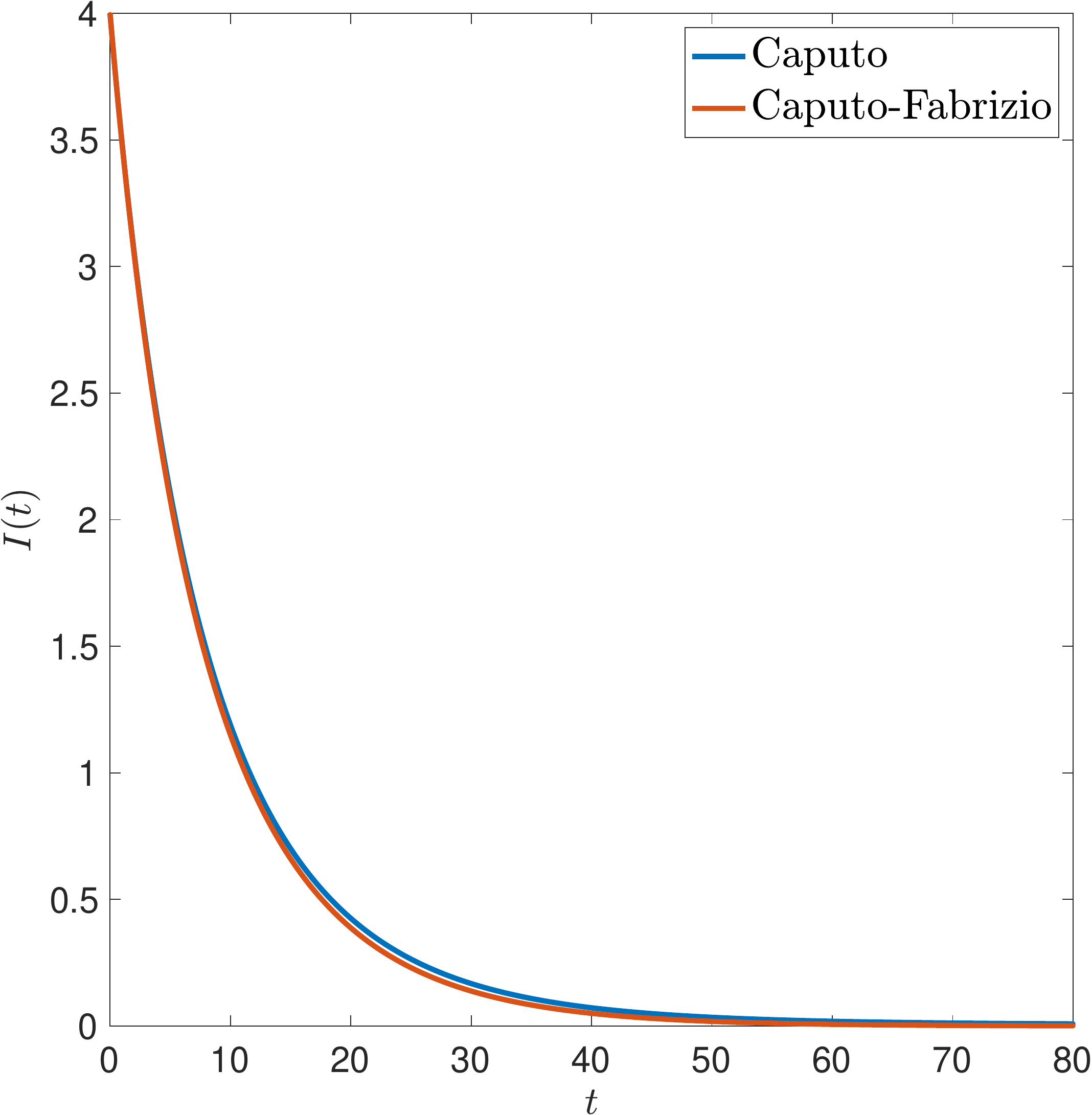}
}
\quad
\subfloat[][$S(t)+I(t)$]{
\includegraphics[width=0.3\columnwidth]{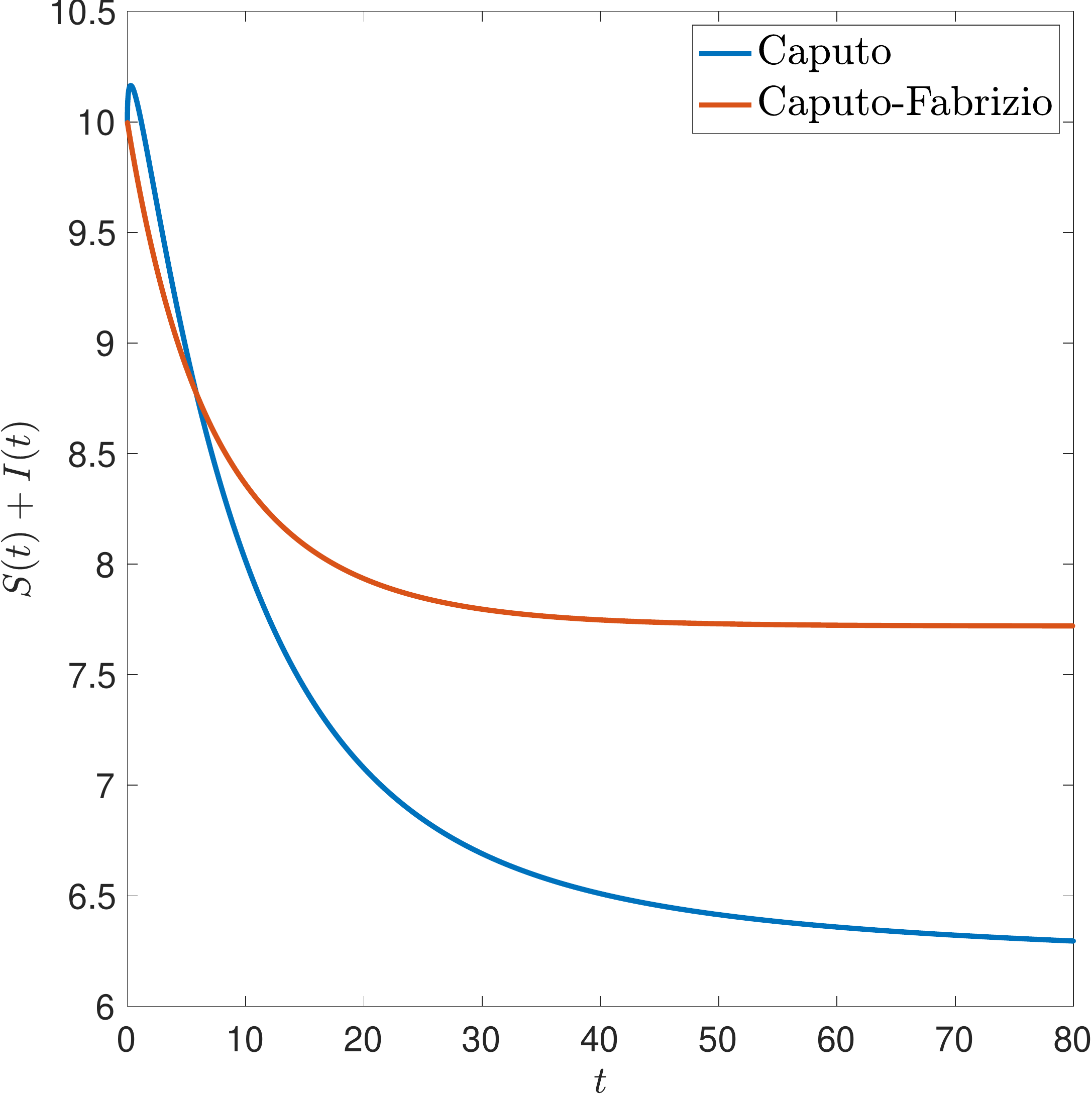}
}\\
\subfloat[][$S(t),\,\alpha=\alpha_{1}=0.8$]{
\includegraphics[width=0.3\columnwidth]{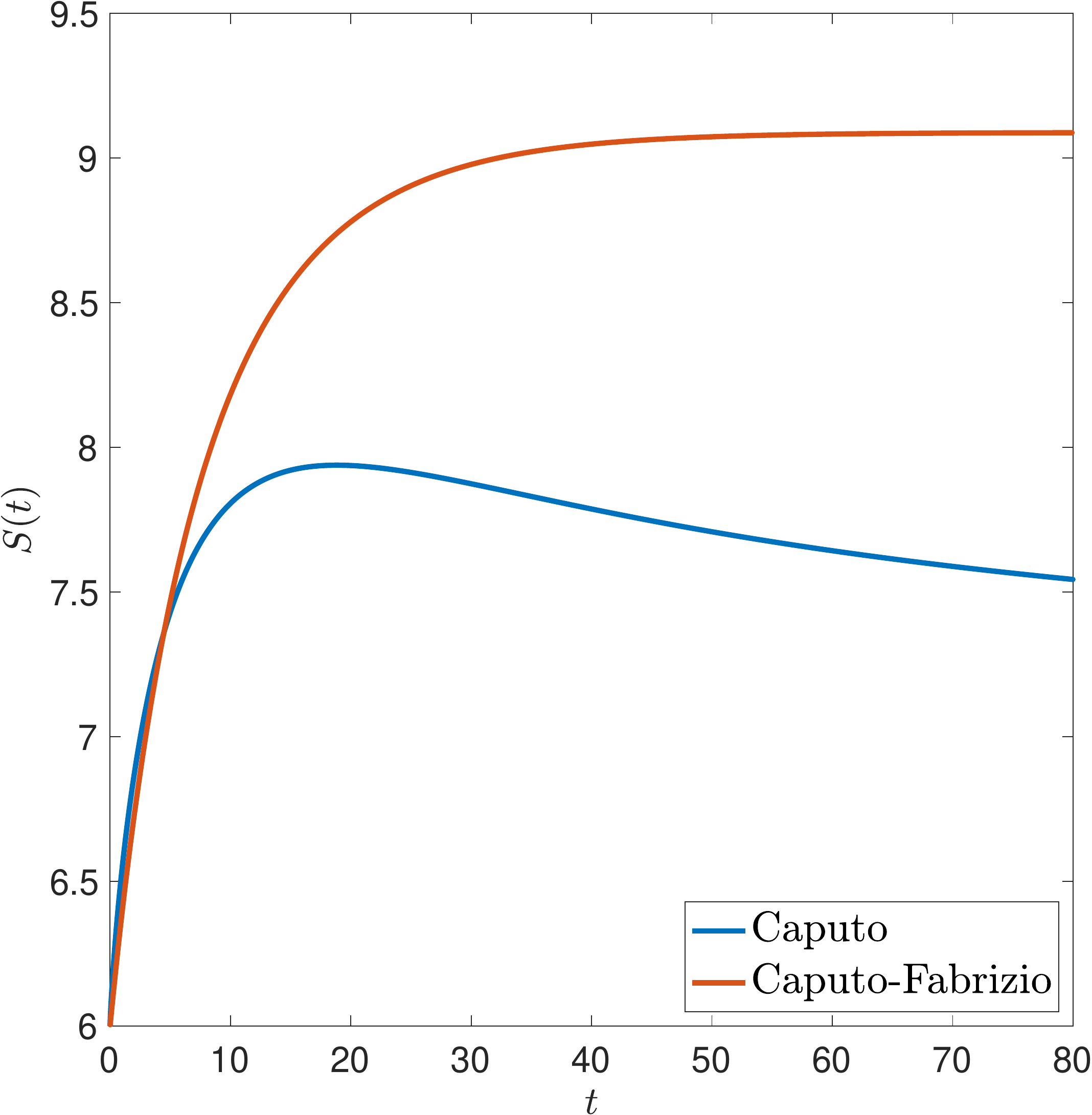}
}
\quad
\subfloat[][$I(t),\, \alpha_{2}=1$]{
\includegraphics[width=0.3\columnwidth]{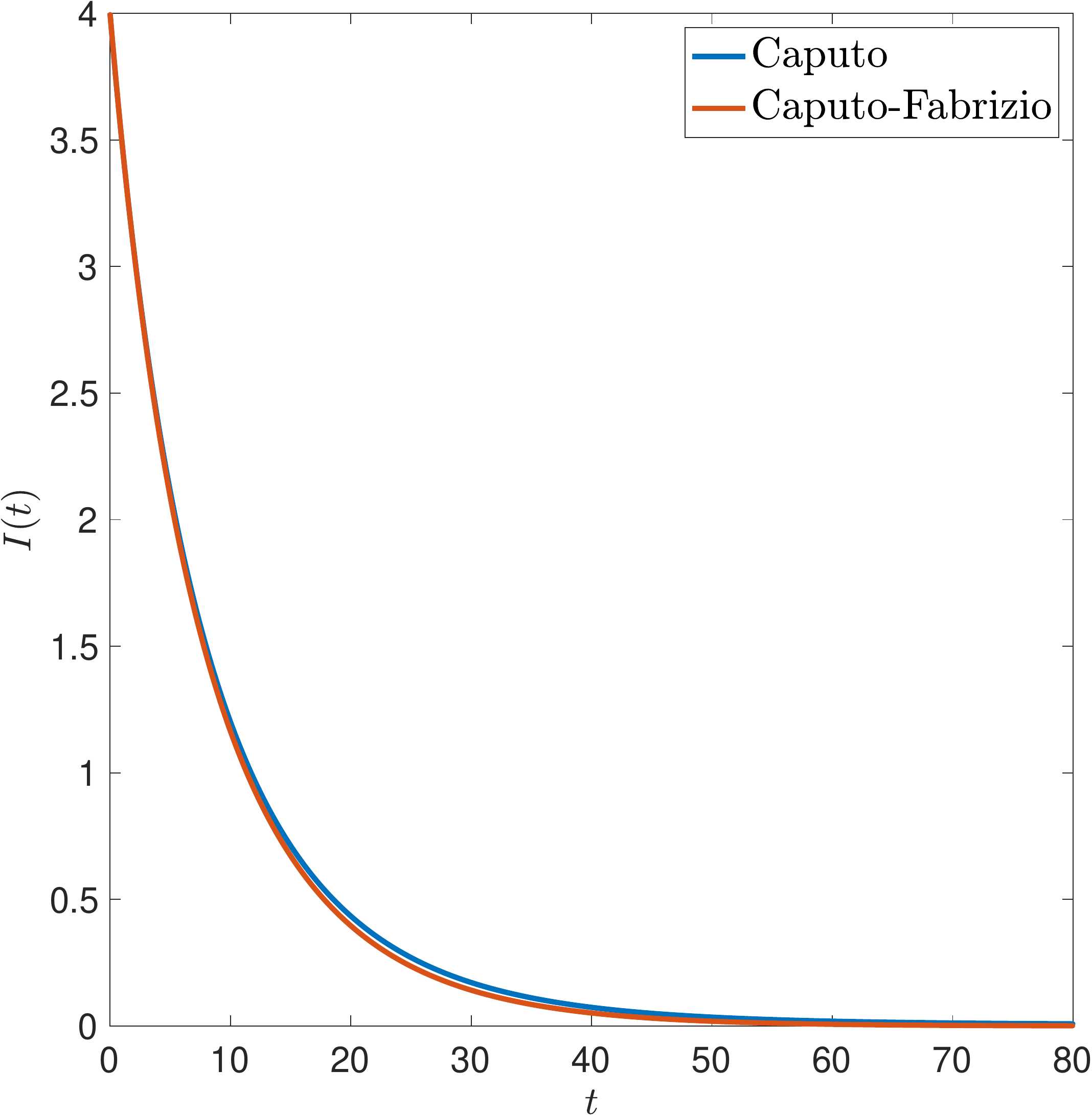}
}
\quad
\subfloat[][$S(t)+I(t)$]{
\includegraphics[width=0.3\columnwidth]{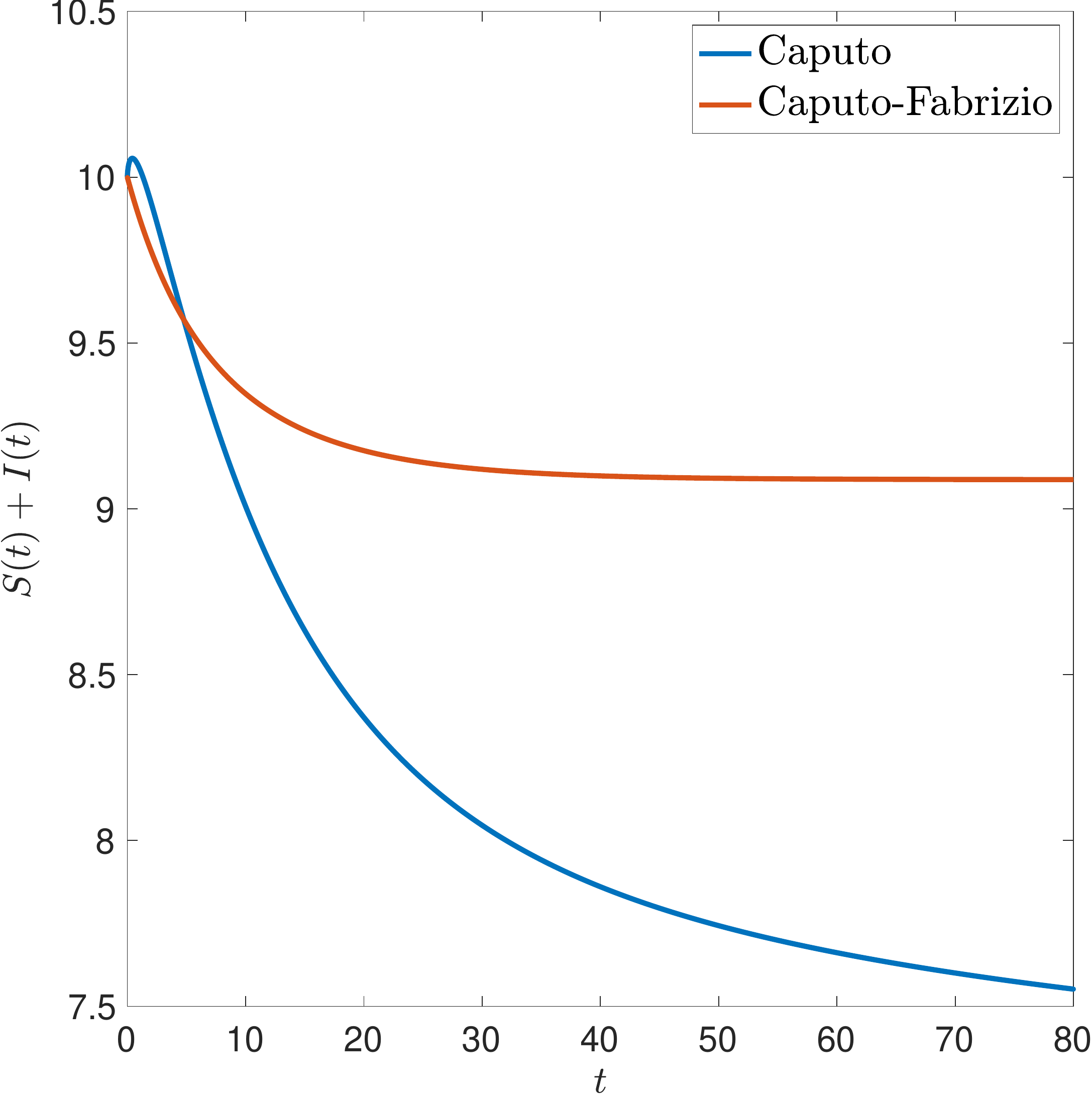}
}
\caption{Numerical solutions   to \eqref{modello3} and \eqref{SIS} for $\alpha_{1}=0.2$ (top), $\alpha_{1}=0.5$ (center) and $\alpha_{1}=0.8$ (bottom) with $\alpha_{2}=1$, $\beta=0.1$, $\gamma=0.2$, $S_{0}=6$, $I_{0}=4$ and $M(\alpha)\equiv1$.}
\label{fig:CFCC2}
\end{figure}

We conclude this section with some tests pointing out the effect of the choice of a particular fractional operator on SIS models. We directly compare the numerical solution to the system with Caputo fractional SIS model \eqref{modello3} and Caputo-Fabrizio fractional SIS model \eqref{SIS}. To this end we fix the intial data $S_0=6$ and $I_0=4$. We consider the set the fractional orders $\alpha_1$ and $\alpha_2$ in \eqref{modello3}, the fractional order $\alpha$ in \eqref{SIS} and we set $\alpha_1=\alpha$, $\alpha_2=1$, letting $\alpha$ vary in the set $\{0.2,0.5,0.8\}$. Figure \ref{fig:CFCC1} depicts the results in a case in which $\gamma<\beta$, in particular $\gamma=0.2$ and $\beta=0.7$, whereas in Figure \ref{fig:CFCC2} we test a case in which $\beta<\gamma$, in particular $\gamma=0.2$ and  $\beta=0.1$. Note that the evolution of  susceptible population $S(t)$ is mostly affected by the change of differential operator and this effect is augmented by choosing small fractional orders. Finally, note that the Caputo-Fabrizio operator preserves the monotonicity properties of the ordinary SIS case, while the Caputo operator used in \eqref{modello3} yields a more complex structure.

\vskip2cm
\section{Conclusions}\label{s4}
We explored the effects of fractional differential operators on  SIS epidemic models. The main novelty of the present paper consists in  letting the susceptible and the infected population evolve with different orders of fractional differential operators. We presented two fractional SIS models with mixed fractional orders. One of them involves  the Caputo fractional derivative, characterized by a singular, power law kernel. The other proposed model relies on  the recently introduced Caputo-Fabrizio differential operator, which has a non-singular, exponential kernel. For both models  we established existence results for the solutions and we conducted a qualitative analysis by means of numerical simulations. In the case of the Caputo-Fabrizio operator, under the assumption that the fractional behavior is restricted to the susceptible population whereas the infected population evolves according to an ordinary differential equation, we were able to move some further step forward in the analysis of the system. Indeed we characterized the equilibria, noticed their strong dependence on the reproduction number according to classical theory,  and proposed a method for inverse problems. Finally, we numerically, directly compared the proposed Caputo and Caputo-Fabrizio SIS models, in order to let emerge the effects of each particular differential operator on the system.  
 
 We believe that  the possibility of tuning the memory effects in a single compartment of the population (susceptible/infected) by means of ad hoc fractional orders, may provide finer and effective tools in data fitting and mathematical modeling of epidemic dynamics.
 
  Further possible extensions of the present work include the extension of the proposed methods to other epidemic models, for instance the SIR model, and to controlled, mixed fractional epidemic dynamics.

\bibliographystyle{alpha}
\bibliography{cfmisto}

\newcommand{\etalchar}[1]{$^{#1}$}
\begin{thebibliography}{LWLT19}

\bibitem[BDL20]{balzotti2020FF}
Caterina Balzotti, Mirko D’Ovidio, and Paola Loreti.
\newblock Fractional sis epidemic models.
\newblock {\em Fractal Fract.}, 4(3):44, 2020.

\bibitem[Cap08]{Caputo}
Michele Caputo.
\newblock Linear models of dissipation whose {$Q$} is almost frequency
  independent. {II}.
\newblock {\em Fract. Calc. Appl. Anal.}, 11(1):4--14, 2008.
\newblock Reprinted from Geophys. J. R. Astr. Soc. {{\bf{1}}3} (1967), no. 5,
  529--539.

\bibitem[CF15]{CapFab}
Michele Caputo and Mauro Fabrizio.
\newblock A new definition of fractional derivative without singular kernel.
\newblock {\em Progr. Fract. Differ. Appl}, 1(2):1--13, 2015.

\bibitem[CL55]{ODE}
Earl~A Coddington and Norman Levinson.
\newblock {\em Theory of ordinary differential equations}.
\newblock Tata McGraw-Hill Education, 1955.

\bibitem[CLYL21]{chen2021AMM}
Yuli Chen, Fawang Liu, Qiang Yu, and Tianzeng Li.
\newblock Review of fractional epidemic models.
\newblock {\em Appl. Math. Model.}, 97:281--307, 2021.

\bibitem[DL18]{DovLor18}
Mirko D'Ovidio and Paola Loreti.
\newblock Solutions of fractional logistic equations by {E}uler's numbers.
\newblock {\em Phys. A}, 506:1081--1092, 2018.

\bibitem[ES13]{elsaka2013MSL}
HAA El-Saka.
\newblock The fractional-order sir and sirs epidemic models with variable
  population size.
\newblock {\em Math. Sci. Lett.}, 2(3):195, 2013.

\bibitem[GLM20]{giga2019AA}
Yoshikazu Giga, Qing Liu, and Hiroyoshi Mitake.
\newblock On a discrete scheme for time fractional fully nonlinear evolution
  equations.
\newblock {\em Asymptotic Analysis}, 120(1-2):151--162, 2020.

\bibitem[HA20]{higazy2020AEJ}
M~Higazy and Maryam~Ahmed Alyami.
\newblock New caputo-fabrizio fractional order seiasqeqhr model for covid-19
  epidemic transmission with genetic algorithm based control strategy.
\newblock {\em Alexandria Engineering Journal}, 59(6):4719--4736, 2020.

\bibitem[Het89]{Hethcote}
Herbert~W. Hethcote.
\newblock Three basic epidemiological models.
\newblock In {\em Applied mathematical ecology ({T}rieste, 1986)}, volume~18 of
  {\em Biomathematics}, pages 119--144. Springer, Berlin, 1989.

\bibitem[HOEK18]{hassouna2018CSF}
M.~Hassouna, A.~Ouhadan, and E.~H. El~Kinani.
\newblock On the solution of fractional order {SIS} epidemic model.
\newblock {\em Chaos Solitons Fractals}, 117:168--174, 2018.

\bibitem[KM27]{KerMcK}
William~Ogilvy Kermack and Anderson~G McKendrick.
\newblock A contribution to the mathematical theory of epidemics.
\newblock {\em P. R. Soc. Lond. A}, 115(772):700--721, 1927.

\bibitem[KSK{\etalchar{+}}21]{khan2021AIMS}
Sajjad~Ali Khan, Kamal Shah, Poom Kumam, Aly Seadawy, Gul Zaman, and Zahir
  Shah.
\newblock {Study of mathematical model of Hepatitis B under Caputo-Fabrizo
  derivative}.
\newblock {\em {AIMS MATHEMATICS}}, {6}({1}):{195--209}, {2021}.

\bibitem[LN15]{nieto}
Jorge Losada and Juan~J Nieto.
\newblock Properties of a new fractional derivative without singular kernel.
\newblock {\em Progr. Fract. Differ. Appl}, 1(2):87--92, 2015.

\bibitem[LWLT19]{li2019NA}
T.~Li, Y.~Wang, F.~Liu, and I.~Turner.
\newblock Novel parameter estimation techniques for a multi-term fractional
  dynamical epidemic model of dengue fever.
\newblock {\em Numer. Algorithms}, 82(4):1467--1495, 2019.

\bibitem[MSK19]{moore2019ADE}
Elvin~J. Moore, Sekson Sirisubtawee, and Sanoe Koonprasert.
\newblock A {C}aputo-{F}abrizio fractional differential equation model for
  {HIV}/{AIDS} with treatment compartment.
\newblock {\em Adv. Difference Equ.}, pages Paper No. 200, 20, 2019.

\bibitem[UKF{\etalchar{+}}20]{ullah2020DCDSS}
Saif Ullah, Muhammad~Altaf Khan, Muhammad Farooq, Zakia Hammouch, and Dumitru
  Baleanu.
\newblock A fractional model for the dynamics of tuberculosis infection using
  {C}aputo-{F}abrizio derivative.
\newblock {\em Discrete Contin. Dyn. Syst. Ser. S}, 13(3):975--993, 2020.

\bibitem[Ver38]{v1}
Pierre-Fran{\c{c}}ois Verhulst.
\newblock Notice sur la loi que la population suit dans son accroissement.
\newblock {\em Corresp. Math. Phys.}, 10:113--126, 1838.

\bibitem[Ver45]{v2}
Pierre-Fran{\c{c}}ios Verhulst.
\newblock Recherches math{\'e}matiques sur la loi d'accroissement de la
  population.
\newblock {\em Journal des {\'e}conomistes}, 12:276, 1845.

\bibitem[Ver47]{v3}
Pierre-Fran{\c{c}}ois Verhulst.
\newblock Deuxi{\`e}me m{\'e}moire sur la loi d'accroissement de la population.
\newblock {\em M{\'e}moires de l'acad{\'e}mie royale des sciences, des lettres
  et des beaux-arts de Belgique}, 20:1--32, 1847.

\end{thebibliography}
\end{document}